\documentclass[11pt]{article}  
\usepackage{amsmath}
\usepackage{amssymb}
\usepackage{theorem}
\usepackage{euscript}
\usepackage{pstricks}
\usepackage{authblk}
\usepackage{verbatim}
\usepackage{hyperref}

\usepackage{ifthen}
\usepackage{xifthen}
\usepackage{xcolor}
\usepackage{fullpage}
\usepackage[english]{babel}
\usepackage{mathtools}
\usepackage{color}

\hypersetup
{
	colorlinks,
	citecolor=black,
	filecolor=black,
	linkcolor=black,
	urlcolor=magenta
}
\hypersetup{linktocpage} 

\newtheorem{theorem}{Theorem}[section]
\newtheorem{lemma}[theorem]{Lemma}
\newtheorem{definition}[theorem]{Definition}
\newtheorem{corollary}[theorem]{Corollary}
\numberwithin{equation}{section}

\newtheorem{assumption}[theorem]{Assumption}
\newtheorem{remark}{Remark}[section]
\newtheorem{proposition}[theorem]{Proposition}
\numberwithin{equation}{section}
\def\II{\mathbb I}

\def\NN{{\mathbb N}}

\def\II{{\mathbb I}}
\def\CC{{\mathbb C}}
\def\CC{{\mathbb C}}

\def\NN{{\mathbb N}}
\def\RR{{\mathbb R}}

\def\Vv{{\mathbb P}}

\def\Vv{{\mathcal P}}

\def\Bb{{\mathcal B}}
\def\calE{{\mathcal E}}

\def\IIi{{\mathbb I}^\infty}

\def\Bb{{\mathcal B}}

\def\Ii{{\mathcal I}}

\def\Ll{{\mathcal L}}

\def\Qq{{\mathcal Q}}
\def\Ss{{\mathcal S}}
\def\Tt{{\mathcal T}}

\def\Vv{{\mathcal V}}
\def\II{{\mathbb I}}
\def\CC{{\mathbb C}}

\def\NN{{\mathbb N}}
\def\RR{{\mathbb R}}

\def\FF{{\mathbb F}}

\def\supp{\operatorname{supp}}
\def\dv{\operatorname{div}}

\newcommand{\bsnul}{{\boldsymbol 0}}
\newcommand{\Hol}  {{\rm Hol}}
\newcommand{\ba}{{\boldsymbol{a}}}
\newcommand{\bb}{{\boldsymbol{b}}}

\newcommand{\be}{{\boldsymbol{e}}}

\newcommand{\bp}{{\boldsymbol{p}}}

\newcommand{\bs}{{\boldsymbol{\nu}}}

\newcommand{\bx}{{\boldsymbol{x}}}

\newcommand{\by}{{\boldsymbol{y}}}

\newcommand{\brho}{{\boldsymbol{\rho}}}
\newcommand{\bsigma}{{\boldsymbol{\sigma}}}

\newcommand{\bnu}{{\boldsymbol{\nu}}}

\newcommand{\rd}{{\rm d}} 
\newcommand{\rev}{{\rm ev}}

\newcommand{\sett}[1]{\left\{#1\right\}}
\newcommand{\mI}{\mathbb{I}}

\newcommand{\norm}[2]{\left\|{#1}\right\|_{#2}}

\newcommand{\abs}[1]{\left|#1\right|}
\newcommand{\brab}[1]{\left\{#1\right\}}
\newcommand{\brac}[1]{\left(#1\right)}

\newcommand{\frku}{{\mathfrak u}}
\title{\sffamily {Approximation of PDE solution manifolds: 
                  \\
Sparse-grid interpolation and quadrature}
\footnote{
The work of Ch. Schwab is supported in part by the Swiss National Science Foundation (SNSF), 
the work of D. D\~ung, V.K. Nguyen and T.D. Pham is funded by the Vietnam National Foundation for Science and Technology Development (NAFOSTED),  under the Swiss--Vietnam joint research programme, Grant No.  IZVSZ2$_{ - }$229568. A part of the work of D. D\~ung, V.K. Nguyen and D.T. Pham was done when they were working at the Vietnam Institute for Advanced Study in Mathematics (VIASM). They would like to thank VIASM for providing a fruitful research environment and working conditions.
}
}
\author[a]{Dinh D\~ung}
\affil[a]{Information Technology Institute, Vietnam National University, Hanoi
	\protect\\
	144 Xuan Thuy, Cau Giay, Hanoi, Vietnam
	\protect\\
	Email: dinhzung@gmail.com}

\author[b]{Van Kien Nguyen}
\affil[b]{Department of Mathematical Analysis, University of Transport and Communications
	\protect\\	No.3 Cau Giay Street, Lang Ward,
	Hanoi, Vietnam
	\protect\\
	Email: kiennv@utc.edu.vn}

\author[c]{Duong Thanh Pham}
\affil[c]{Vietnamese German University,
	\protect\\ Ring road 4, Quarter 4, Thoi Hoa Ward, Ho Chi Minh City, Vietnam
	\protect\\
	Email: duong.pt@vgu.edu.vn}
	
\author[d]{Christoph Schwab}
\affil[d]{
	Seminar for Applied Mathematics, ETH Z\"urich, Z\"urich, Switzerland
\protect \\ Email: schwab@math.ethz.ch}

\date{\today}
\begin{document}
\maketitle
\abstract{
We study fully-discrete approximations and quadratures of 
infinite-variate functions in abstract Bochner spaces associated with a Hilbert space $X$ 
and an infinite-tensor-product Jacobi measure. 
For target infinite-variate functions  taking values in $X$
which admit absolutely convergent Jacobi generalized polynomial chaos expansions, 
with suitable weighted summability conditions for the coefficient sequences,
we generalize and improve prior results on construction of 
sequences of finite sparse-grid tensor-product polynomial interpolation approximations and quadratures,
based on the univariate Chebyshev points.
For a generic stable  discretization of $X$ 
in terms of a dense sequence $(V_m)_{m \in \NN_0}$ of finite-dimensional subspaces, 
we obtain fully-discrete, \emph{linear approximations}
in terms of so-called sparse-grid tensor-product projectors, 
with convergence rates of approximations   as well
as  of sparse-grid tensor-product quadratures of the target functions. 

We verify the abstract assumptions in two fundamental application settings: 
first,
a linear elliptic diffusion equation with affine-parametric coefficients
and second, abstract holomorphic maps between separable Hilbert spaces with
affine-parametric input data encoding. 
For these settings, as in \cite{ZS20,Dung21}, 
cancellation of anti-symmetric terms
in ultra-spherical Jacobi generalized polynomial chaos expansion coefficients 
implies crucially improved convergence rates of
sparse-grid tensor-product quadrature
with respect to the infinite-tensor-product Jacobi weight,
free from the ``curse-of-dimension''.

Largely self-contained proofs of all results are developed.
Approximation convergence rate results in the present setting 
which are based on construction of neural network surrogates,
for unbounded parameter ranges with Gaussian measures,
will be developed in extensions of the present work.  
\section{Introduction}
\label{sec:Intro}
Recent years have seen development of numerical analysis 
and ``high-dimensional'' approximation, 
i.e., of functions depending
on a large, possibly, infinite number of variables.
These arise, 
in connection with PDEs with uncertain input data from function spaces.
Upon representing such input data in suitable bases, for example 
Fourier-, Wavelet- or Frames (see, e.g., \cite{Tri10B} for a lucid 
presentation of constructions of concrete representation systems in 
a wide range of function spaces), the solutions of the PDE become
infinite-parametric maps which, in turn, 
take values in target Hilbert- or Banach spaces $X$.

Parsimonious numerical approximation of such maps by finite-parametric
surrogates requires, as a rule, \emph{dimension-explicit parametric regularity}
combined with \emph{hierarchic, multi-level approximation} in the 
target function spaces.
One central issue in approximation rate estimates is the so-called
``curse of dimensionality'' (CoD). 
For infinite-parametric maps which result from \emph{holomorphic maps between
function spaces}, it has been shown in recent years, starting with 
\cite{TodorChS07,CDS10,CDS11,CCS15}, that the CoD can be overcome.
The key mathematical insight in these works were suitable summability 
results in coefficient sequences of generalized polynomial chaos (GPC) 
expansions of the parametric solution families. Summability, in turn, 
gives rise to $N$-term approximation rate bounds via Stechkin's lemma.
GPC summability results provide \emph{existence} of finite-parametric
approximations of the parametric solutions in function space $X$, 
but are, generally, not constructive.

Subsequently,
\emph{constructive versions of GPC approximations} 
with $N$-term approximation rates have been developed in the past decade,
for example in \cite{Zec18T,Dung19,Dung21,ZS20,Dung21-v11,AdcEMS13,DD2025a,DNSZ2023} 
and the references there.
Related results addressing numerical quadrature with respect to 
probability measures on the coordinate sequence spaces have been developed in 
\cite{HHPS18,ZS20}: there, convergence rates of the sparse-grid quadratures for semi-discrete 
(i.e. assuming exact evaluations of the function values) setting,
for quadrature w.r. to tensor products of the uniform probability measure
on $(-1,1)$. 
A unified mathematical derivation of \emph{constructive} interpolation approximation 
and numerical integration with spatial discretization in the target space $X$ 
of affine-parametric and, more generally, holomorphic maps between function spaces
is the purpose of the present paper.
\subsection{Existing results}
\label{sec:ExRes}
In the semi-discrete setting (i.e. without discretization in $X$) 
and for integration against tensor products of the uniform measure,
in \cite{ZS20} and later in \cite{Dung21,Dung21-v11}, approximation
rate bounds were developed.
By exploiting symmetry properties 
of the uniform measure on $(-1,1)$, 
cancellations of anti-symmetric moments in the Bochner integrals over 
generalized polynomial chaos (for sort GPC) surrogates were shown in \cite{ZS20} 
to imply higher convergence rates of the corresponding Smolyak quadratures.
The general (non-symmetric) case, still semi-discrete, 
for numerical approximation of the integral 
versus the product Jacobi measure $\mu_{\ba,\bb}$ was 
analyzed in \cite[Thm.~3.1.6]{Zec18T} and \cite[Corollary 6.1]{Dung21}, 
in the semi-discrete case. 

Results on fully-discrete approximations 
were developed in  \cite{BCDC17}, \cite[Sec.~3.2]{Zec18T} for  a Galerkin method, 
and \cite[Thms. 6.1, 6.2]{Dung21}.
Results on fully-discrete polynomial interpolations and quadratures were obtained, among others, 
in \cite{Zec18T,ZDS19,Dung19,Dung21-v11}.
Semi-discrete least-squares approximations, i.e., without discretization of the GPC coefficients,
for parametric PDEs have been studied in \cite{BD2024,CCMNT2015,CM2018}, 
whereas full-discrete least-squares approximations are analyzed for example in \cite{DD2025a}.

Adcock et al. address in \cite{ADMHolo} the approximation rate analysis of 
infinite-parametric, holomorphic functions, 
again in a semi-discrete setting (i.e.,
the coefficients in the approximation are assumed to take values in Hilbert- or Banach spaces $X$).
Algorithmic aspects of localizing, for a given budget of $N$ GPC terms, 
set of at most $N$ multi-indices are addressed in \cite{AdcEMS13}.
Bachmayr et al. in \cite{BCM18} address parametric analyticity of solution families 
of affine-parametric, linear diffusion equation. 
Analytic regularity of parametric solutions is established in \cite{BCM18} 
via a real-variables argument, with inductive proofs to bound parametric derivatives of the solution. 
Nonlinear (adaptive) sparse grid approximations have been investigated in 
\cite{EGSZ14,EGSZ15} and more recently in 
\cite{bachmayr2025adaptivestochasticgalerkinfinite,MBVoulis22}.
\subsection{Contributions}
\label{sec:Contr}
We  generalize and improve  the convergence rate results in 
\cite{HHPS18,ZS20,Dung21-v11} to multi-level, 
sparse-grid polynomial interpolation and quadrature 
based on univariate Chebyshev nodes and with respect to general Jacobi weights, 
as considered in \cite{Zec18T,Dung21}. 

The crucial improvement in the convergence rate 
of the sparse-grid quadrature is verified 
for more general weights than considered in \cite{ZS20},
specifically for tensorized versions of the univariate (symmetric) ultra-spherical weights, 
both in the fully-discrete and semi-discrete case.
The construction of finite-parametric sparse-grid approximations for polynomial interpolation and quadrature
is via a weighted thresholding of index sets, similar to \cite{Zec18T,Dung21-v11}. 
The construction will be used in a subsequent part of this work to construct neural network 
approximations of parametric PDE solution manifolds with convergence rates.
\subsection{Layout}
\label{sec:Layout}
In Sec.~\ref{sec:FlDscInt}, we recap basic terminology and definitions on Jacobi
orthogonal polynomials, their infinite tensor products and corresponding GPC expansions,
and univariate interpolation operators. 
We develop constructions of sparse-grid tensor-product polynomial interpolation 
and quadrature for approximation of $X$-valued, 
parametric functions in Bochner space associated with a Hilbert space $X$ 
and a Jacobi infinite tensor product measure on the parameter sequences.
We establish convergence rates of fully-discrete polynomial interpolations 
as well as of the related, fully-discrete quadratures 
for the target functions under some weighted summability condition on Jacobi GPC coefficients, 
and subject to discretization in $X$. 
In Sec.~\ref{sec:Expl}, 
we apply the abstract results in the preceding section to two important examples: 
a linear elliptic diffusion equation with affine-parametric coefficients
and abstract holomorphic maps between separable Hilbert spaces with affine-parametric input encoding.}
%
\subsection{Notation}
\label{sec:Notat}
$\NN = \{1,2,3,\ldots\}$ denotes the natural numbers, i.e. the set of positive integers
and we write $\NN_0 = \NN \cup \{0\}$.
An important role in GPC expansion will be played by the 
set of ``finitely supported'' multiindices
$\FF = \big\{ \bs = (\nu_j)_{j\in \NN} : \bs\in \NN_0^{\NN}, \ \sum_{j \in \NN} \nu_j < \infty \big\}$.
Observe that $\FF$ is countable.
We denote by $\bsnul\in \FF$ the zero multi-index, and by $\be_i$ the sequence
$(\delta_{ij})_{j \in \NN}$.
We introduce in $\FF$ a half-ordering via 
$$
\bs \leq \bs' \; \Longleftrightarrow \forall j\in \NN:\;\; \nu_j\leq \nu'_j.
$$
A multi-indexed sequence 
$(\sigma_\bs)_{\bs \in \FF} \subset \RR$ 
is called \emph{increasing} if 
$\sigma_{\bs'} \le \sigma_\bs$ for $\bs' \le \bs$. 

For $\bs = (\nu_j)_{j\in \NN} \in \FF$, 
we introduce for $0<p<\infty$ 
\[ |\bs|_p := \Bigg( \sum_{j \in \NN} \nu_j^p \Bigg)^{1/p}.
\]
We also set ${\rm supp}(\bs) = \{ j\in \NN: \nu_j \ne 0 \}$, and 
\[
|\bs|_0 := \#(\{j\in \NN: \nu_j \ne 0 \}),\quad 
|\bs|_\infty := \max_{j \in \NN} \nu_j 
\;.
\]
We shall also use the following notation: for $\kappa \in \NN$,
\begin{equation} \nonumber
	\FF_\kappa := \{\bs \in \FF: \nu_j \in \NN_{0,\kappa},  \ j \in \NN \},
	\quad\mbox{where}\quad 
	\NN_{0,\kappa} := \{n \in \NN_0: n = 0, \kappa , \kappa + 1, \ldots\}.
\end{equation}
Obviously $\FF = \FF_1$.
To describe certain cancellations in Jacobi expansion due to symmetry,
in our analysis of sparse-grid quadrature we shall require the even index set
\begin{equation*}\label{eq:Feven}
\FF_\rev  :=\{\bs\in  \FF: \nu_j\in 2\NN_0 \text{ for all } j\in \NN\} \subset \FF_2.
\end{equation*}
A multi-index set $\Lambda \subset \FF$ $(\Lambda \subset \FF_\rev)$  
is called \emph{downward closed in $\FF$} (resp. in $\FF_\rev$) 
if the inclusion $\bs \in \Lambda$ implies $\bs' \in \Lambda$ 
for every $\bs' \in \FF$ ($\bs' \in \FF_\rev$) such that $\bs' \le \bs$.   

Throughout, $\II:= [-1,1] $, and $\II^\infty = [-1,1]^\infty$ denotes 
the countable cartesian product. 

\section{Fully discrete  approximations in Bochner spaces}
\label{sec:FlDscInt}
In this section, we develop fully discrete sparse-grid GPC interpolation and quadrature for infinite-variate functions in Bochner spaces associated with  a Hilbert space $X$ and an infinite-tensor-product Jacobi measure.
We consider infinite-variate $X$-valued functions with weighted $\ell_2$-summability 
conditions for the Jacobi GPC expansion coefficient sequences, 
and and a generic, stable  discretization of $X$ 
in terms of a dense sequence $(V_m)_{m \in \NN}$
of finite-dimensional subspaces  with certain approximation properties. 
We present  finite sparse-grid, tensor-product polynomial interpolations 
and sparse, Smolyak-type quadrature rules built from the univariate Chebyshev nodes.
\subsection{Jacobi polynomials}
\label{sec:JacPol}
For given $a,b > -1$, 
let  $(J_k)_{k\in \NN_0}$ be the sequence of 
(probabilistic) Jacobi polynomials on $\II = [-1,1] $ 
which are normalized with respect to the Jacobi probability measure $\mu_{a,b}$ 
on $\II$ endowed with the sigma algebra of Borel sets $\Bb(\II)$ on $\II$, 
i.e., 
$$
\int_{\II} |J_k(y)|^2 \rd \mu_{a,b}(y) =\int_{\II} |J_k(y)|^2 \delta_{a,b}(y) \rd y =1, \ \
k\in \NN_0,
$$
where the Jacobi weight function $\delta_{a,b}(y)$ in $(-1,1)$ 
is given by
\[
\delta_{a,b}(y):=c_{a,b}(1-y)^a(1+y)^b, \quad
c_{a,b}:=\frac{\Gamma(a+b+2)}{2^{a+b+1}\Gamma(a+1)\Gamma(b+1)}.
\]
In particular, $(\II,\Bb, \mu_{a,b})$ is a probability space,
and
the Jacobi polynomials normalized in this way are $\mu_{a,b}$-orthonormal, 
i.e.,
\begin{equation}\label{eq:Jort}
\int_{\II} J_k(y) J_l(y) \rd \mu_{a,b}(y) = \delta_{kl}\;,\quad k,l \in \NN_0 \;.
\end{equation}
In particular, $J_0 \equiv 1$ for all $a,b>-1$.

Important examples contained in this setting are:
(i) $a=b=0$, when $\mu_{a,b}$ is the uniform probability measure on $\II$
with $ c_{0,0}=1/2$ and $\delta_{a,b} \equiv 1/2$, and $(J_k)_{k \in \NN_0}$
are the Legendre polynomials,
(ii) $a=b=-1/2$ which corresponds to the family of the Chebyshev polynomials, 
and (iii) $a=b>-1$ which corresponds to the family of ultra-spherical (Gegenbauer) polynomials.

In all cases, one has the Rodrigues' formula
\begin{equation*}\label{eq:Rodrig}
J_k(y ) 
\ = \
\frac{c_k^{a,b}}{k! 2^k}(1-y)^{-a}(1+y)^{-b} \frac{\rd^k}{\rd y^k} 
\left((y^2-1)^k(1-y)^a(1+y)^b\right),
\end{equation*}
where $c_0^{a,b}:= 1$ and
\begin{equation}\label{eq-cabk}
	c_k^{a,b}
	:= \
	\sqrt{\frac{(2k+a+b+1)k! \Gamma(k+a+b+1) \Gamma(a+1) \Gamma(b+1)}
		{\Gamma(k+a+1)\Gamma(k+b+1)\Gamma(a+b+2)}}, \ k \in \NN.
\end{equation}

From \cite[Theorem 7.32.1]{Szego1939} and the relations
 $J_0\equiv 1 $ and $c_k^{a,b}\sim k^{1/2}$, $k \in \NN$, 
by a direct computation one can derive the bound 
\begin{align} \label{J_s<-1} 
	\norm{J_k}{L_\infty(\II)} \leq (1+\lambda_{a,b}k)^{\max\sett{a,b,-1/2}+1/2}
\end{align}
for $k\in \NN_0$, where $\lambda_{a,b}$ is a positive constant  independent of $k$.

\subsection{Jacobi chaos}
\label{sec:JacGPC}
Multivariate Jacobi polynomials are constructed by tensorization.
For GPC expansions, arbitrary large tensor products
of univariate polynomials are required. To describe these, we 
introduce suitable notation.
Let 
$ \ba = \brac{a_j}_{j \in \NN} \in \ell_\infty(\NN)$ 
and 
$ \bb = \brac{b_j}_{j \in \NN} \in \ell_\infty(\NN)$ 
with 
$-1 < \underline{a}, \underline{b}$, 
where
$\underline{a} : = \inf_{{j \in \NN}} a_j$, $\underline{b} := \inf_{j \in \NN} b_j$.
We define the infinite-dimensional Jacobi probability measure $\mu$ on $\IIi$ as the  
tensor product of the Jacobi probability measures $\mu_{a_j,b_j}$:
\begin{equation} \label{eq:ProdJPM}
	\mu_{\ba,\bb} :=\bigotimes_{j \in \NN} \mu_{a_j,b_j}.
\end{equation}
When $\ba,\bb$ are clear from the context, 
we also write $\mu$ instead of $\mu_{\ba,\bb}$.

For $\bs=(\nu_j)_{j \in \NN} \in \FF$ and $\by=(y_j)_{j \in \NN} \in \IIi$
define the tensor product Jacobi polynomial
\begin{equation}\label{eq:ProdJac}
	J_\bs(\by):=\bigotimes_{j \in \NN}J_{\nu_j}(y_j).
\end{equation}
Due to $J_0 = 1$ for any $\bs\in \FF$ 
the product in \eqref{eq:ProdJac}
contains only $|\bs|_0$-many nontrivial factors. 
The univariate orthonormality \eqref{eq:Jort} 
then implies with Fubini's theorem 
\begin{equation*}\label{eq:Jorti}
\forall \bs,\bs' \in \FF: \quad 
\int_{\IIi} J_\bs(\by) J_{\bs'}(\by) \rd \mu(\by) = \delta_{\bs \bs'} 
\;.
\end{equation*}
Hence, 
the countable collection $(J_\bs)_{\bs \in \FF}$ 
is an orthonormal basis of $L_2(\IIi;\mu)$.

For a summability index $0 < p \le \infty$, 
we introduce the Bochner space $\Ll_p(X) := L_p(\IIi, X;\mu)$ 
as the set of all strongly $\mu$-measurable functions $\IIi \to X$ 
taking values in a Hilbert space $X$, 
equipped with the (quasi-)norm  
\begin{equation*} \label{Bochner space}
	\|v\|_{\Ll_p(X)}
	:= \
	\begin{cases}
		\left(\int_{\IIi} \|v(\by)\|_X^p \, \rd \mu(\by) \right)^{1/p}, \ \ & 0 < p < \infty,
         \\
		\operatorname{ess \ sup}_{\by \in \IIi} \|v(\by)\|_X, \ \ &p = \infty.
	\end{cases}
\end{equation*}
There hold the norm inequalities for $0 < p_1 < p_2 \le \infty $,
\begin{equation} \label{Norm ineqalities}
	\|\cdot\|_{\Ll_{p_1}(X)} \le \|\cdot\|_{\Ll_{p_2}(X)}.	
\end{equation}
Let $C(\IIi, X)$ be the Banach space of  
all functions  defined on $\IIi$ taking values in $X$, 
which are continuous on $\IIi$  
w.r. to the product topology. 
According to the Tychonoff theorem (see, e.g.,  \cite[page 143: Thm. 13]{Ke1955}), 
this topology renders $\IIi$ compact.
A norm in $C(\IIi, X)$ is then defined by 
\begin{equation}
\nonumber
	\|v\|_{C(\IIi, X)}:= \max_{\by \in \IIi} \|v(\by)\|_X.
\end{equation}
Note that 
$\norm{v}{C(\IIi, X)}= \norm{v}{\Ll_\infty(X)}$ for $v \in C(\IIi, X)$.


If $v \in \Ll_2(X)$ for a Hilbert space $X$, the formal
\emph{Jacobi generalized polynomial chaos (GPC) expansion} 
of $v$ reads
\begin{equation} \label{J-series}
v = \sum_{\bs\in\FF} v_\bs J_\bs, 
\text{ where} \ \
v_\bs:=\int_{\IIi} v(\by)J_\bs(\by) {\rd\mu} (\by),
\end{equation}
with the equality and convergence in the Hilbert space $\Ll_2(X)$. 
There holds the Parseval's identity
\begin{equation} \label{eq-Parseval}
	\|v\|_{\Ll_2(X)}^2 = \sum_{\bs\in\FF} \|v_\bs\|_X^2.
\end{equation}
We remark that the
$\Ll_2$-convergence implied by \eqref{eq-Parseval} 
does not imply absolute convergence. 
For absolute convergence we need a certain additional condition 
on weighted summability as stated in Lemma~\ref{l:I lambda xi}.
We start by introducing the weights that we shall consider.

For  $\theta, \lambda\ge0$ we define the set 
$\bp(\theta,\lambda):= \brac{p_\bs(\theta,\lambda)}_{\bs \in \FF}$ by
\begin{equation}\label{[p_s]}
	p_\bs(\theta,\lambda):=\prod_{j\in \NN}(1+\lambda \nu_j)^\theta,\ \ \bs\in \FF.
\end{equation}
We use also the abbreviation: $\bp(\theta):=\bp(\theta,1)$.  
From \eqref{J_s<-1}  we can see that
\begin{equation}\label{|J_bs|<}
	\norm{J_\bs}{L_\infty(\IIi)}
	\ \le \
	p_\bs(\theta_0,\lambda_0),\ \ \bs\in \FF,
\end{equation}
where 
\begin{equation}\label{lambda_0}
\lambda_0:= \sup_{j \in \NN}\lambda_{a_j,b_j} \ < \ \infty
\end{equation} 
with the constants $\lambda_{a_j,b_j}$ as in \eqref{J_s<-1}, and 
\begin{equation}\label{theta_0}
\theta_0 := \ \max\brab{{\sup_{j\in \NN} a_j, \sup_{j\in \NN}b_j},-1/2} + 1/2 \ \ge \ 0.
\end{equation}
\subsection{Sparse-grid tensor-product polynomial interpolation}
\label{sec:SpGIntrp}
In this section,
we construct linear fully discrete (multi-level) sparse-grid tensor-product polynomial interpolations 
for approximation of functions taking values in the Banach space $X^2\subset X^1$,
with weighted $\ell_2$-summability of Jacobi GPC expansion coefficients.
We consider Hilbert spaces $X^1$ and $X^2$ satisfying a certain 
``spatial" approximation property which we formalize in Assumption~\ref{assum2}, item~(iii) 
below.

For $m\in \NN_0$, we denote by $Y_m$ the set of Chebyshev nodes 
\begin{equation}\label{eq:ChebNod} 
	Y_m 
	:=
	\sett{ y_{m,k} = -\cos \frac{(2k+1)\pi}{2(m+1)}: k = 0,1,2,\ldots, m}.
\end{equation}
If $v$ is a function on $\RR$ taking values in a Hilbert space $X$ and $m\in \NN_0$, 
we define the function $I_m(v)$ on $\RR$ taking values in $X$ by
\begin{equation}\label{Imv Lmk}
	I_m(v)
	:=
	\sum_{k = 0}^m 
	v(y_{m,k})
	L_{m,k},
	\quad 
	L_{m,k}(y)
	:=
	\prod_{j=1,\ldots,m, j\not=k} 
	\frac{y - y_{m,j}}{y_{m,k}-y_{m,j}}.
\end{equation}
The function $I_m(v)$ interpolates $v$ at $y_{m,k}$, i.e., $I_m(v)(y_{m,k}) = v(y_{m,k})$ for $k = 0,\ldots, m$.

The Lebesgue constant is given by 
\[
\lambda_m(Y_m)
:=
\sup_{\norm{v}{C(\mI)}\le 1}
\norm{I_m(v)}{C(\mI)}.
\]
There holds the inequality 
\begin{align*}
	\lambda_m(Y_m)
	\leq  
	1+\frac{2}{\pi}\log(m+1),
\end{align*}
see, for examples, \cite[eq. (10)]{Brutman97}. 
Hence we get
\begin{equation} 	\label{lambda Ym Eu}
	\lambda_m(Y_m)  < \log(2m+3).
\end{equation}

We define the univariate increment operator 
$\Delta^{{\rm I}}_m$ for $m \in \NN_0$ 
by
\begin{equation} \label{Delta m def}
	\Delta^{{\rm I}}_m
	:= \
	I_m - I_{m-1},
\end{equation} 
with the convention $I_{-1} = 0$, 
and the univariate even increment operator $\Delta^{{\rm I}*}_m$ 
for $m \in 2\NN_0$ 
by
\begin{equation} \label{Delta m* def}
	\Delta^{{\rm I}*}_m
	:= \
	I_m - I_{m-2},
\end{equation} 
with the convention $I_{-2} = 0$. 
From \eqref{lambda Ym Eu} it follows that
\begin{align} \label{norm{Delta_m(v)}}
	\norm{ \Delta^{{\rm I}*}_m(v)}{L_\infty(\mI)} , \
        \norm{\Delta^{{\rm I}  }_m(v)}{L_\infty(\mI)}
	\ \le \
	2\log (2m+3)\,\norm{v}{{L_\infty(\mI)}}, \ \ v \in C(\II), \ \ m \in \NN_0.
\end{align}

Recalling~\eqref{Delta m def}, 
for  a function $v$ defined on $\IIi$ and taking values in a Hilbert space $X$, 
we introduce the tensor product operator $\Delta^{{\rm I}}_\bs$, $\bs \in \FF$, by
\begin{equation}\label{Delta s def}
	\Delta^{{\rm I}}_\bs(v)
	:= \
	\bigotimes_{j \in \NN} \Delta^{{\rm I}}_{\nu_j}(v),
\end{equation}
where the univariate operator
$\Delta^{{\rm I}}_{\nu_j}$ is applied to the univariate function 
$\bigotimes_{i < j} \Delta^{{\rm I}}_{\nu_i}(v)$ by considering it as a 
function of  variable $y_j$ with the other variables held fixed. 
For a finite set $\Lambda \subset \FF$, the 
sparse tensor-product interpolation operator $I_{\Lambda}$ is defined by 
\begin{align}
	I_{\Lambda}
	:=
	\sum_{\bs\in \Lambda} 
	\Delta^{{\rm I}}_{\bs}.
	\label{I Lambda xi def}
\end{align}

For $\bs \in \FF$, define $R_{\bs}:=\sett{\bs'\in\FF: \bs'\le \bs}$.
Here the inequality $\bs' \le \bs$ means that $\nu_j' \le \nu_j$, $j \in \NN$.

\begin{assumption} \label{assum1}
	$v \in \Ll_2(X)$ and there exist a set of positive numbers 
	$ (\sigma_\bs)_{\bs\in \FF}$ strictly larger than $1$  and a number $0 < q <2$  such that 
	\begin{equation*} 
		\left(\sum_{\bs\in\FF} \big(\sigma_\bs \|v_\bs\|_{X}\big)^2\right)^{1/2}  \le M  <\infty \ \ \text{and} \ \ 
		\brac{\sum_{\bs\in\FF}\brac{p_\bs(\theta,\lambda)^{2/q}\sigma_\bs^{-1}}^q}^{1/q}  \le K < \infty.
	\end{equation*}
\end{assumption}

\begin{lemma}\label{l:I lambda xi}
	Let $\varepsilon$ be a fixed positive number and $C_\varepsilon$ such that
	$$
	2(1 + \lambda_0 k)^{\theta_0} \log(2k+3) \leq (C_\varepsilon k+1)^{\theta_0+\varepsilon}, \ \ \forall k\in \NN_0,
	$$
with $\theta_0, \lambda_0$ given in  \eqref{lambda_0} and \eqref{theta_0}.

	Let $v\in \Ll_2(X)$ and satisfy Assumption~\ref{assum1} 
        with
	\begin{equation}\label{theta,lambda}
		\theta:= \theta_0 + 1+ \varepsilon ,\ \ \lambda:=C_\varepsilon+1.
	\end{equation}
	Then the function $v$ can be identified with an element in $C(\IIi, X)$. 
	
	Additionally, for every
	$\by\in \IIi$ we can represent $v(\by)$ by the series
	\begin{equation*}\label{theseries}
		v(\by)= \sum_{\bs\in\FF} v_\bs J_\bs(\by), 
	\end{equation*}
        with absolute convergence in $X$. 
The series \eqref{J-series} converges unconditionally in $\Ll_2(X)$ to $v$.
\end{lemma}

\begin{proof} We first prove that the series in \eqref{J-series}  converges absolutely in $C(\IIi, X)$.
	By \eqref{|J_bs|<} and \eqref{theta,lambda} we get
	\begin{equation*}\label{abs conv}
		\begin{aligned}
			\sum_{\bs\in\FF} \|v_\bs J_\bs\|_{C(\IIi, X)} 
			& \leq 
			\sum_{\bs\in\FF} \|v_\bs\|_X\|J_\bs\|_{{L_\infty(\IIi)}}
			\leq 
			\sum_{\bs\in\FF} \|v_\bs\|_X p_\bs (\theta,\lambda)
			\\
			&\leq  \brac{ 
				\sum_{\bs\in \FF}
				\big(\sigma_{\bs} \norm{v_{\bs}}{X}\big)^2
			}^{1/2}
			\brac{ 
				\sum_{\bs\in \FF}
				\Big(p_{\bs}(\theta,\lambda)
				\sigma_{\bs}^{-1}\Big)^2
			}^{1/2} 
			\\
			&\leq M
			\brac{\sum_{\bs\in\FF}\brac{p_\bs(\theta,\lambda)^{2/q}\sigma_\bs^{-1}}^q}^{1/q}
			<\infty.
		\end{aligned}
	\end{equation*}
	Here and below, $\theta$ and $\lambda$ are as in \eqref{theta,lambda}.
	Since $C(\IIi, X)$ is a Banach space, the series in \eqref{J-series}  
        converges  absolutely and therefore, unconditionally  in the Banach space $C(\IIi, X)$, 
        and hence, due to the norm inequalities \eqref{Norm ineqalities}, 
        in $\Ll_2(X)$ to an element $\bar{v} \in C(\IIi, X) \subset \Ll_2(X)$. 
	In particular, we have for every $\by \in \FF$, 
	$$\bar{v}(\by)= \sum_{\bs\in\FF} v_\bs J_\bs(\by)$$
	with the absolute convergence  in $X$.
	Since $\bar{v}$ and $v$ have the same Jacobi GPC expansion we get $v=\bar{v}$ in  $\Ll_2(X)$. 
        Hence $v(\by)=\bar{v}(\by)$ $\mu$-almost everywhere.  
        This means that the function  $v$ can be treated as an element in $C(\IIi, X)$.  
	\hfill
\end{proof}

For the sparse-grid interpolation approximation bounds, 
we require \emph{a sparsity hypothesis on $v$}.
As in previous works \cite{CDS10,BCDC17,Zec18T,Dung21,Dung21-v11}, 
given Hilbert spaces $X^1$ and $X^2$ with $X^2 \subset X^1$, 
sparsity in these spaces here takes the form of  what we call ``double-weighted 
summability'' of coefficients in Jacobi GPC expansions of $v\in \Ll_2( X^2)$.
To construct linear, fully discrete approximation methods, 
besides weighted $\ell_2$-summabilities with respect to $X^1$ and $X^2$ 
we need an approximation property on the spaces $X^1$ and $X^2$. 
Combining these requirements,
we say that $v$ satisfies Assumption~\ref{assum2} iff
\begin{assumption}\label{assum2}
	\ \
\begin{itemize}
	\item[{\rm (i)}] {[Hilbert scale]}	
	$X^1$ and $X^2$ are Hilbert spaces, 
        $X^2$ is  a linear subspace of  $X^1$ and 
        there exists a constant $C>0$ such that for all $w\in X^2$ holds
            $\|w\|_{X^1} \le C\, \|w\|_{X^2}$;

        \item[{\rm (ii) }] {[GPC representation]}
        $v \in \Ll_2(X^2)$ is represented by the series
	\begin{equation} \label{J-SeriesX}
		v = \sum_{\bs\in\FF} v_\bs J_\bs, \quad v_\bs \in X^2,
	\end{equation}	
	\item[{\rm (iii)}]  {[Stability and consistency of spatial approximation]}
	There exist 
        a sequence $(V_m)_{m \in \NN_0}$ of subspaces $V_m \subset X^1 $ 
        of dimension at most $n$ with $V_0 = \brab{0}$, 
        and 
        a sequence $(P_m)_{m \in \NN_0}$ of linear operators from $X^1$ into $V_m$, 
	and a number $\alpha>0$ such that $P_0(w)=0$ and there exist constants $C_1, C_2> 0$
        such that
		\begin{equation} \label{SpatialAppr}
			\|P_m(w)\|_{X^1} \leq 
			C_1 \|w\|_{X^1} , \quad
			\|w-P_m(w)\|_{X^1} \leq 
			C_2 m^{-\alpha} \|w\|_{X^2}, \quad \forall m \in \NN_0, \quad \forall w \in X^2.
		\end{equation}
		\item[{\rm (iv)}] [Double-weighted Jacobi-summability]
		for $i=1,2$,  there exist numbers $q_i$ with $0 < q_1\leq q_2 <\infty$ and $q_1 < 2$,  
                and families $\bsigma_i := (\sigma_{i;\bs})_{\bs \in \FF} \subset (1,\infty)$ 
                such that 
		\begin{equation} \label{weighted summability}
			\sum_{\bs\in\FF} \big(\sigma_{i;\bs} \|v_\bs\|_{X^i}\big)^2 \le M_i <\infty \quad  \text{and} \quad 
			\left(p_{\bs}(\theta,\lambda)^{2/q_i}\sigma_{i;\bs}^{-1}\right)_{\bs \in \FF} \in \ell_{q_i}(\FF)
		\end{equation}		
		with $\theta ,\lambda$ as in \eqref{theta,lambda}. 
	\end{itemize}
\end{assumption}

Spaces $X^1$ and $X^2\subset X^1$ 
will typically belong to a suitable scale of Sobolev or Besov spaces 
describing regularity of parametric maps in ``physical'' co-ordinates.
This in turn provides approximation rate $\alpha$ of projections $P_m$.
Assumption~\ref{assum2} stipulates the regularity and approximation rate
bounds to hold \emph{uniformly} with respect to $\by\in \II^\infty$
for \emph{one} collection $\{ V_m \}_{m \in \NN_0}$,
i.e., $V_m$ is independent of $\by$.

For each $k\in \NN_0$, 
with $P_m$ as in Assumption~\ref{assum2}, we define for $w \in X^2$
\begin{align*}
	\delta_k(w)
	:=
	P_{2^k}(w) -  P_{2^{k-1}}(w),
	\quad 
	k\in \NN, 
	\quad 
	\delta_0:= P_0(w).
\end{align*}
We have from \eqref{SpatialAppr}
\begin{equation} \label{del k1}
	\norm{\delta_k(w)}{X^1}
	\le 
	C 2^{-\alpha k} \norm{w}{X^2},
	\quad 
	k\in \NN_0.
\end{equation}

For $w\in X^2$ satisfying Assumption~\ref{assum2}, item (iii),
we can represent $w$ by the series
\begin{align}\label{[delta-seriesw]}
	w
	=
	\sum_{k=0}^\infty
	\delta_{k}(w),
\end{align}
with equality and unconditional convergence in $X^1$.

In the setting of Assumption~\ref{assum2},
for a finite set $G\subset \NN_0\times \FF$, 
we define the approximation space 
\begin{align*}
	\Vv(G)
	&
	:=
	\Bigg\{ 
		v = \sum_{(k,\bs)\in G} v_k  J_{\bs}: 
		v_k\in V_{2^k}
	\Bigg\}.
	\label{VG def}
\end{align*}
The linear projector  $\Ss_{G}: \Ll_2( X^2)  \rightarrow \Vv(G)$ is then defined by 
\begin{align*}
	\Ss_G v
	:=
	\sum_{(k,\bs)\in G} 
	\delta_k(v_{\bs}) J_{\bs}
	\quad 
	\text{for }
	v = \sum_{\bs\in \FF} v_{\bs} J_{\bs}\in \Ll_2( X^2),
	\quad 
	v_{\bs}\in X^2.
\end{align*}
For the constructive sparse-grid tensor product interpolation and finite truncation of  GPC expansion, 
    as in e.g. \cite{Zec18T,ZS20,ZDS19,Dung21,Dung21-v11}, 
    index sets $G$ are chosen by thresholding.
\begin{definition}\label{def:G(xi)} 
{[Thresholded index sets]}
	For a threshold parameter $\xi >1$, for the approximation rate $\alpha$ 
        and the summability exponents $q_1,q_2$ as in Assumption~\ref{assum2}, 
        define
	\begin{equation} \label{vartheta:=}
		\tau:=\frac{2\alpha}{2-q_2},\qquad	
                \vartheta:= \frac{2 }{2-q_2}\bigg(\frac{1}{q_1}-\frac{1}{2}\bigg),
		\qquad \eta: = \bigg(\frac{1}{q_1}-\frac{1}{2}\bigg)^{-1},
	\end{equation}
	and the thresholded index set
	\begin{align}\label{G(xi)}
		G(\xi)
		:=
		\begin{cases}
			\sett{
				(k,\bs)\in \NN_0\times \FF: 2^k \sigma_{2;\bs}^{q_2} \le \xi
			}
			&
			\text{if } \alpha \le 1/q_2 - 1/2,
			\\
			\sett{
				(k,\bs)\in \NN_0\times \FF: \sigma_{1;\bs}^{q_1}\le \xi(\log \xi)^\eta, \ 
				2^{\tau k} \sigma_{2;\bs}  \le \xi
				^\vartheta}
			&
			\text{if } \alpha > 1/q_2 - 1/2.
		\end{cases}
	\end{align}
\end{definition}

We will need the following auxiliary result on 
fully discrete approximation in the Bochner space $\Ll_p(X^1)$ 
for the operator $\Ss_{G(\xi)}$.

\begin{lemma}\label{lemma[L_p-approx]}
	Let $v$ satisfy Assumption~\ref{assum2} with summability exponents $q_1,q_2$ 
        and let $\alpha > 0$ be as in Assumption~\ref{assum2}, item (iii). Let further the threshold index sets $G(\xi)$ be as in Definition~\ref{def:G(xi)}, 
        for a threshold parameter $\xi>1$.

	Then there exists a constant $C>0$ such that for every threshold parameter $\xi >1$ and 
        for every $0<p\leq 2$ holds
	\begin{equation*} 	\label{L_p-rate}
		\|v-\Ss_{G(\xi)}v\|_{\Ll_p(X^1)} 
		\ \leq \
		C  
		\begin{cases}
			\xi^{-\alpha}
			 & \text{ if }  \ \alpha \le 1/q_2 - 1/2, 
                        \\
			\xi^{-(1/q_1 - 1/2)}
			  & \text{ if }  \ \alpha > 1/q_2 - 1/2.
		\end{cases}	
	\end{equation*}
\end{lemma}
	\begin{proof} In this proof the positive constant $C$ may change its value from place to place but is independent of $\xi$. Because of the inequality
		$\|\cdot\|_{\Ll_p(X^1)} \le  \|\cdot\|_{\Ll_2(X^1)}$,  it is sufficient to prove the theorem for $p=2$.
		Under Assumption~\ref{assum2}, in a fashion analogous to the proof of \cite[Lemma 3.4]{Dung21-v11}, 
                the function $v$ can be represented as the series 
		\begin{equation} \label{[delta-seriesJ]}
			v
			\ = \
			\sum_{(k,\bs) \in \NN_0 \times \FF} \delta_k (v_\bs)\, J_\bs
		\end{equation}
		converging absolutely and hence, unconditional  in $\Ll_2(X^1)$ to $v$. 
		
		We first consider the case $\alpha \le 1/q_2 - 1/2$.
		We have by Parseval's identity and the unconditional convergence of the 
                series \eqref{[delta-seriesJ]} that
		\begin{equation} \nonumber
			\begin{split}
				\|v - \Ss_{G(\xi)} v\|_{\Ll_2(X^1)}^2
				\ &= \
				\Bigg\|\sum_{(k,\bs) \in \NN_0 \times \FF} \delta_k (v_\bs) \, J_\bs - 
				\sum_{\bs \in \FF}  \sum_{2^k \sigma_{2;\bs}^{q_2} \le \xi} \delta_k (v_\bs) \, J_\bs\Bigg\|_{\Ll_2(X^1)}^2 
				\\[1.5ex]
				\ &= \
				\Bigg\|\sum_{\bs \in \FF}  \ \sum_{\xi \sigma_{2;\bs}^{-q_2}<2^k } \delta_k (v_\bs) \, J_\bs\Bigg\|_{\Ll_2(X^1)}^2 
				\ = \
				\sum_{\bs \in \FF}  \ \Bigg\|\sum_{\xi \sigma_{2;\bs}^{-q_2}<2^k } \delta_k (v_\bs)\Bigg\|_{X^1}^2 
				\\[1.5ex]
				\ &\le \
				\sum_{\bs \in \FF}  \ \Bigg(\sum_{\xi \sigma_{2;\bs}^{-q_2}<2^k } \|\delta_k (v_\bs)\|_{X^1}\Bigg)^2
				\ \le \
				\sum_{\bs \in \FF}  \ \Bigg(\sum_{\xi \sigma_{2;\bs}^{-q_2}<2^k } C \, 2^{-\alpha k}\|v_\bs\|_{X^2}\Bigg)^2 
				\\[1.5ex]
				\ &\le \
				C \, \sum_{\bs \in \FF} \|v_\bs\|_{X^2}^2 \ \Bigg(\sum_{2^k >\xi \sigma_{2;\bs}^{-q_2}}  2^{-\alpha k}\Bigg)^2
				\ \le \
				C \, \sum_{\bs \in \FF} \|v_\bs\|_{X^2}^2 \ (\xi \sigma_{2;\bs}^{-q_2})^{-2\alpha}. 
			\end{split}
		\end{equation}
		Hence, by  the inequalities  $q_2\alpha \le 1$ and $\sigma_{2;\bs} >1$, and 
		\eqref{weighted summability} we derive that
		\begin{equation} \nonumber
			\begin{split}
				\|v - \Ss_{G(\xi)} v\|_{\Ll_2(X^1)}^2
				\ &\le \
				C \, \xi^{- 2\alpha} \sum_{\bs \in \FF} 
				\big(\sigma_{2;\bs} \|v_\bs\|_{X^2}\big)^2 
				\ \leq \
				\ C \, \xi^{- 2\alpha},
			\end{split}
		\end{equation}
		which proves the lemma for the case $\alpha \le 1/q_2 - 1/2$.
		
		Let us consider the case $\alpha > 1/q_2 - 1/2$. 
                Putting
		\[
		v_\xi
		:=
		\sum_{\{ \bs : \sigma_{1;\bs}^{q_1} \le \xi (\log\xi)^\eta\} } v_\bs J_\bs,
		\]
		we get
		\begin{equation*} \label{triangle-ineq}
			\|v- \Ss_{G(\xi)} v\|_{\Ll_2(X^1)}
			\ \le \
			\|v- v_\xi\|_{\Ll_2(X^1)} + \|v_\xi - \Ss_{G(\xi)} v\|_{\Ll_2(X^1)}.
		\end{equation*}
		By Lemma~\ref{l:I lambda xi} the series \eqref{J-series} converges unconditionally in $\Ll_2(X^1)$ to $v$.
		Hence,  employing Assumption~\ref{assum2}, item (iv), 
                the norm $\|v- v_\xi\|_{\Ll_2(X^1)}$  can be estimated by
		\begin{equation*} \label{|u-u_xi|}
			\begin{split}
				\|v- v_\xi\|_{\Ll_2(X^1)}^2 
				\ &= \ 
				\sum_{\sigma_{1;\bs}^{q_1} > \xi (\log\xi)^\eta} \|v_\bs\|_{X^1}^2
				\ = \
				\sum_{\sigma_{1;\bs}^{q_1} > \xi (\log\xi)^\eta } \sigma_{1;\bs}^{-2} \big(\sigma_{1;\bs}\|v_\bs\|_{X^1}\big)^2 
				\\[1.5ex] 
				\ &\le \big(\xi (\log\xi)^\eta\big)^{-2/q_1}
				\sum_{\bs \in \FF}  (\sigma_{1;\bs}\|v_\bs\|_{X^1})^2 
				\ \le  C \big(\xi (\log\xi)^\eta\big)^{-2/q_1}.
			\end{split}
		\end{equation*}
		
		For the norm $\|v_\xi - \Ss_{G(\xi)} v\|_{\Ll_2(X^1)}$, with 
		$N= N(\xi,\bs):= 2^{\big\lfloor\log_2\big(  \sigma_{2;\bs}^{-1/\tau} 
			\xi^{\vartheta/ \tau}\big)\big\rfloor}$ 
		we have 
		\begin{equation} \nonumber
			\begin{split}
				\|v_\xi - \Ss_{G(\xi)} v\|_{\Ll_2(X^1)}^2
				\ &= \
				\sum_{\sigma_{1;\bs}^{q_1}  \le  \xi (\log\xi)^{\eta} }
				\Bigg\| v_\bs - \sum_{2^k \le \sigma_{2;\bs}^{-1/\tau}  \xi^{\vartheta/\tau}}  \delta_k (v_\bs) \Bigg\|_{X^1}^2  
				\\[1.5ex]
				\ &= \
				\sum_{\sigma_{1;\bs}^{q_1}  \le  \xi(\log\xi)^\eta }  \Big\| v_\bs - P_N (v_\bs) \Big\|_{X^1}^2 
				\ \le \
				C\, \sum_{\sigma_{1;\bs}^{q_1}  \le  \xi(\log\xi)^\eta } N^{-2\alpha} \| v_\bs \|_{X^2}^2 
				\\[1.5ex]
				\ &\le \
				C\, \sum_{\sigma_{1;\bs}^{q_1}  \le  \xi(\log\xi)^\eta } 
				\big(\sigma_{2;\bs}^{-1/\tau}  \xi^{\vartheta/\tau}\big)^{-2\alpha} \| v_\bs \|_{X^2}^2 
				\\& = \ 
				C\, \xi^{- 2\vartheta \alpha/\tau}
				\sum_{\sigma_{1;\bs}^{q_1}  \le  \xi(\log\xi)^\eta }\sigma_{2;\bs}^{2\alpha/\tau}\| v_\bs \|_{X^2}^2 
				\\[1.5ex]
				\ &= \ 
				C\, \xi^{-2(1/q_1 - 1/2)}
				\sum_{\sigma_{1;\bs}^{q_1}  \le  \xi(\log\xi)^\eta }\sigma_{2;\bs}^{2-q_2}\| v_\bs \|_{X^2}^2 
				\\[1.5ex]
				\ &\le \ 
				C\, \xi^{-2(1/q_1 - 1/2)}
				\sum_{\bs \in \FF}\big(\sigma_{2;\bs}\| v_\bs \|_{X^2}\big)^2 
				\ \le \ 
				C\, \xi^{-2(1/q_1 - 1/2)}.
			\end{split}
		\end{equation} 
		Here we used the equalities $\vartheta \alpha/\tau = 1/q_1 - 1/2$, 
		\ $2\alpha/\tau = 2 - q_2$ 
		and Assumption~\ref{assum2}, item~(iv).
		Summing up,  we find 
		\begin{equation} \nonumber
			\begin{split}
			\|v- \Ss_{G(\xi)} v\|_{\Ll_2(X^1)}
				\le
				C\, \xi^{- (1/q_1 - 1/2)}
			\end{split}
		\end{equation}
in the case $\alpha>1/q_2-1/2$.		\hfill
	\end{proof}
	
	\begin{definition}\label{def:IG}
		Given an index set $G \subset \NN_0 \times \FF$ 
		with the structure \eqref{G(xi)}, 
		we introduce the \emph{sparse tensor product interpolation operator}
		$\Ii_G: C(\IIi, X^2) \to \Vv(G)$
		by
		\begin{equation} \label{eq:SPTI}
			\Ii_G v
			:= \
			\sum_{(k,\bs) \in G} (\delta_k \otimes \Delta^{{\rm I}}_\bs) (v).
		\end{equation}
		Here, the sparse-grid, tensor-product interpolation increments 
                $\Delta^{{\rm I}}_\bs$ are as in \eqref{Delta s def}.
	\end{definition}
	
	The sparse-grid interpolation operator $\Ii_G v$ 
	corresponds to a linear (i.e. non-adaptive), fully discrete polynomial interpolation approximation.
	It is constructed by a sum over the index set $G$, 
	of anisotropic tensor products of dyadic, successive differences of spatial  approximations to $v$, 
	and of successive differences of  tensorized Lagrange interpolating polynomials. 
	
	The symmetry of the univariate Jacobi probability measures $\mu_{a_j,b_j}$ 
		in the ultra-spherical case when $a_j=b_j>-1$, $j \in \NN$, 
        implies the cancellation
	\begin{equation}\label{eq:JacSym}
		\int_{\IIi} J_{\bs} (\by) \rd\mu (\by) = 0  \quad \mbox{when there exists $j$ such that $\nu_j$ is odd}.
	\end{equation}
	A corresponding set of \emph{symmetric sparse tensor-product interpolators} 
	on $C(\IIi, X^2)$ exploits these cancellations, and will be relevant in Section~\ref{Integration}
	below 
	for the corresponding sparse-grid quadratures, 
	as observed first in \cite{Zec18T,ZS20} in the Legendre case and 
        later in \cite{Dung21,Dung21-v11} in the Jacobi case. 
	\begin{definition}\label{def:IGev}
		[Symmetric sparse tensor product interpolator] 
		For a finite index set  $G \subset \NN_0 \times \FF_\rev$ 
		with the structure \eqref{G(xi)},
		the interpolation operators 
		$$
		\Ii_G^*: C(\IIi, X^2) \to \Vv(G)
		$$ 
		as defined as in \eqref{eq:SPTI}, 
		with the tensorized increments $\Delta^{{\rm I}_\bs}$ for $\bs \in \FF_\rev$
		replaced by $\Delta^{{\rm I}*}_\bs := \bigotimes_{j \in \NN} \Delta^{{\rm I}*}_{\nu_j}$,
		with $\Delta^{{\rm I}*}_{\nu_j}$ defined as in \eqref{Delta m* def}.
	\end{definition}

\begin{theorem}\label{thm: FullyPolAppr}
	[Sparse-grid tensor-product interpolation convergence]
	Suppose that $v$ satisfies Assumption~\ref{assum2}.
		
		Then for the index set $G(\xi)$ in Definition~\ref{def:G(xi)} and 
		for each $n\in \NN$ there exists a number $\xi_n$ such that 
		$\dim \Vv(G(\xi_n)) \le  n$.	
		Furthermore, 
		there exists a constant $C>0$ such that 
		for any $0< p \le 2$ and any $n\in \NN$,
                we have 
		for the sparse-grid tensor-product interpolation operator
		$$
		\Ii_{G(\xi_n)}: C(\IIi, X^2) \to \Vv(G(\xi_n)),
		$$ 
                the error bound
	\begin{align}
		\norm{v-\Ii_{G(\xi_n)}v}{\Ll_p( X^1)} 
		\le 
		C
		\begin{cases}
			n^{-\alpha} &\text{if } \alpha \leq  1/q_2 - 1/2,
			\\
			n^{-\beta}(\log n)^\kappa &\text{if } \alpha >  1/q_2 - 1/2,
		\end{cases} 
		\label{vSGxi3}
	\end{align}
	where $\alpha$ is the convergence rate given by~\eqref{SpatialAppr} and 
	\begin{equation}\label{beta}
		\beta :=\brac{\frac{1}{q_1} - \frac{1}{2}}
		\frac{\alpha}{\alpha +\delta},\ \ \ \delta:=\frac{1}{q_1}-\frac{1}{q_2}, \ \ \
		0 < \kappa := \
		\frac{\alpha +1/2 - 1/q_2}{\alpha +1/q_1 - 1/q_2} \ < \ 1.
	\end{equation}	
\end{theorem}

\begin{proof} 
This theorem 
is proved along the lines of the proof of \cite[Theorem 3.1]{Dung21-v11}.
We provide details for completeness. 
It is sufficient to prove the theorem for $p=2$. 	
From the condition \eqref{weighted summability} in Assumption~\ref{assum2}  
it follows that the series \eqref{J-SeriesX} converges unconditionally  
in $\Ll_2( X^1)$ to $v$ by Lemma \ref{l:I lambda xi}. In this proof the constant $C$ may change its value from place to place but is always independent of $\xi$.

\medskip
\noindent Step 1: \emph{Relation of the sparse-grid interpolant $I_{\Lambda}v$ to 
          the truncated Jacobi gpc expansion.}
We have that $\Delta^{{\rm I}}_\bs J_{\bs'} = 0$ for every $\bs \not\le \bs'$. 
If  $\Lambda \subset \FF$ is a downward closed set in $\FF$, 
then $I_{\Lambda} J_\bs = J_\bs$ for every $\bs \in \Lambda$, 
and hence we can write
	\begin{equation}\label{I_Lambda}
		I_{\Lambda}v
		\ = \
		I_{\Lambda}\bigg(\sum_{ \bs \in \FF} v_\bs \,J_\bs \bigg)
		\ =  \
		\sum_{ \bs \in \FF} v_\bs \,I_{\Lambda} J_\bs
		\ =  \
		\sum_{ \bs \in \Lambda}  v_\bs \, J_\bs
		\ + \ \sum_{\bs \not\in \Lambda} v_\bs \, I_{\Lambda \cap R_\bs}\, J_\bs.
	\end{equation}

	Let $\xi>1$ be given. 
        For $k \in  \NN_0$, put
	\begin{equation} \nonumber
		\Lambda_k(\xi)
		:= \ 
		\begin{cases}
			\big\{\bs \in \FF: \,\sigma_{2;\bs}^{q_2} \leq 2^{-k}\xi\big\} \quad & \text{ if }  \ 
			\alpha \le 1/q_2 - 1/2;\\
			\big\{\bs \in \FF: \, \sigma_{1;\bs}^{q_1} \le \xi (\log \xi)^\eta, \  
			\sigma_{2;\bs} \leq 2^{- \tau k}\xi^\vartheta \big\} \quad  & \text{ if }  \ 
			\alpha > 1/q_2 - 1/2.
		\end{cases}
	\end{equation}
Define further
$$
k(\xi):= \left\{ \begin{array}{l} \lfloor \log_2 \xi \rfloor \;\;\mbox{if}\;\; \alpha \le 1/q_2 - 1/2,
\\
	\lfloor \vartheta \tau^{-1}\log_2 \xi \rfloor \;\; \mbox{if} \;\;\alpha > 1/q_2 - 1/2.
\end{array}\right.
$$
	Observe that $\Lambda_k(\xi) = \emptyset$ for all $k > k(\xi)$, and consequently, 
	we have that
	\begin{equation}  \label{Eq[Ii]}
		\Ii_{G (\xi)} v
		\ = \ 
		\sum_{k=0}^{k(\xi)} \delta_k \Bigg(\sum_{ \bs \in \Lambda_k(\xi)} \Delta^{{\rm I}}_\bs \Bigg)v
		\ = \ 
		\sum_{k=0}^{k(\xi)} \delta_k I_{\Lambda_k(\xi)}v.
	\end{equation}
	Since the sequence $(\sigma_{2;\bs})_{\bs \in \FF}$ is increasing,  
        the index sets $\Lambda_k(\xi)$ are downward closed sets in $\FF$ and, consequently, 
        the sequence $\big\{\Lambda_k(\xi)\big\}_{k=0}^{k(\xi)}$ is nested in the inverse order, 
        i.e., 
        $\Lambda_{k'} \subset \Lambda_k(\xi)$ if $k' > k$, 
        and 
	$\Lambda_0$ is the largest and $\Lambda_{k_0} = \{0_\FF\}$ for some $0\leq k_0\leq k(\xi)$ and $\Lambda_k=\emptyset$ if $k>k_0$.

	From the unconditional convergence of the series \eqref{[delta-seriesw]} to $v$,  
        and from \eqref{Eq[Ii]} and \eqref{I_Lambda} we derive that
	\begin{equation}\nonumber
		\begin{split}
			\Ii_{G (\xi)} v
			\ &= \
			\sum_{k=0}^{k(\xi)}  \sum_{ \bs \in \Lambda_k(\xi)} \delta_k(v_\bs) \, J_\bs
			\ + \ 
			\sum_{k=0}^{k(\xi)}  \sum_{\bs \not\in \Lambda_k(\xi)} \delta_k(v_\bs) 
			I_{\Lambda_k(\xi) \cap R_\bs}\, J_\bs
			\\[1.5ex]
			\ &= \
			\Ss_{G (\xi)} v
			\ + \ 
			\sum_{k=0}^{k(\xi)}  \sum_{\bs \not\in \Lambda_k(\xi)} \delta_k(v_\bs) 
			I_{\Lambda_k(\xi) \cap R_\bs}\, J_\bs.
		\end{split}
	\end{equation}
	This implies that
	\begin{equation} \label{v-Ii_G}
		v \ - \ \Ii_{G (\xi)} v
		\ = \
		v \ - \ \Ss_{G (\xi)} v
		\ - \ 
		\sum_{k=0}^{k(\xi)}  \sum_{\bs \not\in \Lambda_k(\xi)} \delta_k(v_\bs) 
		I_{\Lambda_k(\xi) \cap R_\bs}\, J_\bs.
	\end{equation}
	Observe that for $k \le k(\xi)$, if $\bs \not\in\Lambda_k(\xi)$, 
        then $(k,\bs) \not\in G(\xi)$.  
        Hence, by \eqref{v-Ii_G} it follows that 
	\begin{equation} \label{[|u-Iu|<]1}
		\big\|v- \Ii_{G (\xi)} v\big\|_{\Ll_2(X^1)}
		\ \le \
		\big\|v- \Ss_{G (\xi)} v\big\|_{\Ll_2(X^1)} 
		+  
		\ \sum_{(k,\bs) \not\in G (\xi)} \|\delta_k(v_\bs)\|_{X^1} \, 
		\big\|I_{\Lambda_k(\xi) \cap R_\bs}\, J_\bs\big\|_{L_2(\IIi,\mu)}.
	\end{equation}
	
\noindent
Step 2:        Next, \emph{we claim}
\begin{equation} \label{I_Lambdap}
	\norm{ 
		I_{\Lambda_k(\xi)\cap R_{\bs}}
		\brac{
			J_{\bs}
	}}{L_2(\mI^\infty,\mu
		)}
	\le 
	p_{\bs}(\theta, \lambda),
\end{equation}
with $\theta$ and $\theta$ being given in \eqref{theta,lambda}.
This can be proven in a manner analogous to the proof of \cite[Eqn.~(3.26)]{Dung21-v11}.  
We recall the definition~\eqref{I Lambda xi def} of the sparse-grid interpolant $I_\Lambda$, 
and have 
	\begin{align}
		\norm{ 
			I_{\Lambda_k(\xi)\cap R_{\bs}}
			\brac{
				J_{\bs}
		}}{L_\infty(\mI^\infty)}
		\le 
		\sum_{\bs'\in \Lambda_k(\xi)\cap R_{\bs}}
		\norm{ 
			\Delta^{{\rm I}}_{\bs'}
			\brac{
				J_{\bs}
		}}{L_\infty(\mI^\infty)}.
		\label{I Lambda 1}
	\end{align}
	Since $\bs'\leq \bs$, 
        from \eqref{norm{Delta_m(v)}} 
        we derive that
	\[
	\norm{ \Delta^{{\rm I}}_{\nu'_j}\brac{J_{\nu_j}}}{L_\infty(\mI)}
	\le 
	2\log(2\nu_j'+3)
	\norm{J_{\nu_j}}{L_\infty(\mI)}
	\leq 2(1 + \lambda_0 \nu_j)^{\theta_0} \log(2\nu_j+3),
	\]
with $\theta_0$ and $\lambda_0$ being given in \eqref{lambda_0} and \eqref{theta_0}, respectively.
	Hence, we have 
	$$
	\norm{ 
		\Delta^{{\rm I}}_{\nu_j'}
		\brac{J_{\nu_j}}}{L_\infty(\mI)}
	\leq
	(C_\varepsilon \nu_j+1)^{\theta_0+\varepsilon}.
	$$
	This, together with~\eqref{I Lambda 1},
        gives 
	\begin{align*} \nonumber
		\norm{ 
			I_{\Lambda_k(\xi)\cap R_{\bs}}
			\brac{
				J_{\bs}
		}}{L_\infty(\mI^\infty)}
		&
		\le
		\sum_{\bs'\in \Lambda_k(\xi)\cap R_{\bs}} \prod_{j\in \supp(\bs)}
		(C_\varepsilon \nu_j+1)^{\theta_0+\varepsilon}
		\notag
		\\
		&
		\le 
		\abs{R_{\bs}} \prod_{j\in \supp(\bs)}
		(C_\varepsilon \nu_j+1)^{\theta_0+\varepsilon}
		\notag
		\le 
		p_{\bs}(1)\,p_{\bs}(\theta_0+\varepsilon,C_\varepsilon)
		\notag
		\\
		&
		\le 
		p_{\bs}(\theta_0 + 1 +\varepsilon,C_\varepsilon+1)
		= p_\bs(\theta, \lambda).
		\label{I Lambda Js1}
	\end{align*}
	This proves \eqref{I_Lambdap}.
	
\noindent Step 3: 
        From  \eqref{[|u-Iu|<]1}  and \eqref{I_Lambdap} it follows that 
	\begin{equation} \label{[|v-Iv|}
		\big\|v- \Ii_{G (\xi)} v\big\|_{\Ll_2(X^1)}
		\ \le \
		\big\|v- \Ss_{G (\xi)} v\big\|_{\Ll_2(X^1)} 
		+  
		A(\xi),
	\end{equation}
	where
	\begin{equation} \label{A(xi)}
		A(\xi):= \ \sum_{(k,\bs) \not\in G (\xi)} \|\delta_k(v_\bs)\|_{X^1}\cdot 
		p_\bs(\theta, \lambda). 
	\end{equation}
	In the next steps, 
        we use the inequality \eqref{[|v-Iv|} 
        to establish bounds for $\big\|v- \Ii_{G (\xi)} v\big\|_{\Ll_2(X^1)}$.

\noindent Step 4: {\it The case $\alpha \le 1/q_2  - 1/2$.}	
	Lemma  \ref{lemma[L_p-approx]}  gives
	\begin{equation}  \label{v-Ss_Gv}
		\Big\|v- \Ss_{G (\xi)} u\Big\|_{\Ll_2(X^1)}
		\ \le \ 
		C\,  \xi^{-\alpha}.
	\end{equation}
	Let us estimate the term $A(\xi)$ in \eqref{A(xi)} which appears 
        in the right-hand side of \eqref{[|v-Iv|}. 
Bounding $\| \delta_k(v_{\bnu}) \|_{X^1}$ in~\eqref{A(xi)} with~\eqref{del k1} 
        we derive that
	\begin{equation} \nonumber
		\begin{split}
			A(\xi)
			\ &\le \
			C	\sum_{(k,\bs) \not\in G (\xi)} 2^{-\alpha k}p_\bs(\theta, \lambda) \|v_\bs\|_{X^2}
			\ = \
			C \sum_{\bs \in \FF} p_\bs(\theta, \lambda)  \|v_\bs\|_{X^2} \ \sum_{2^k >\xi \sigma_{2;\bs}^{-q_2}}  2^{-\alpha k}
			\\[1.5ex]
			\ &\le \
			C \sum_{\bs \in \FF} p_\bs(\theta, \lambda)  \|v_\bs\|_{X^2} \ \big(\xi \sigma_{2;\bs}^{-q_2}\big)^{-\alpha}
			\ \le \
			C \xi^{- \alpha}\, \sum_{\bs \in \FF} p_\bs(\theta, \lambda)  \sigma_{2;\bs}^{q_2\alpha}\|v_\bs\|_{X^2}.
		\end{split}
	\end{equation}
	By  the inequalities  $2(1 - q_2\alpha) \ge q_2$ and $\sigma_{2;\bs} >1$ and the assumptions we have that
	\begin{equation} \nonumber
		\begin{split}
			\sum_{\bs \in \FF} p_\bs(\theta, \lambda)  \sigma_{2;\bs}^{q_2\alpha}\|v_\bs\|_{X^2}
			\ &\le \
			\left(\sum_{\bs \in \FF}  (\sigma_{2;\bs}\|v_\bs\|_{X^2})^2\right)^{1/2}
			\left(\sum_{\bs \in \FF} p_\bs(\theta, \lambda)^2 \sigma_{2;\bs}^{-2(1 - q_2\alpha)}\right)^{1/2}
			\\[1.5ex]
			\ &\le \
			\left(\sum_{\bs \in \FF}  \big(\sigma_{2;\bs}\|v_\bs\|_{X^2}\big)^2\right)^{1/2}
			\left(\sum_{\bs \in \FF} p_\bs(\theta,\lambda)^2\sigma_{2;\bs}^{- q_2}\right)^{1/2}
			 < \infty.
		\end{split}
	\end{equation}	
	Thus, we obtain in \eqref{A(xi)}
	\begin{equation} \nonumber
		A(\xi)
		\ \le \
		C \xi^{-\alpha}.
	\end{equation}
	This together with \eqref{[|v-Iv|} and \eqref{v-Ss_Gv} 
        implies that
	\begin{equation} \nonumber
		\|v- \Ii_{G(\xi)} u\|_{\Ll_2(X^1)}
		\  \le \ C \xi^{-\alpha}.
	\end{equation}
We also have by \eqref{weighted summability}
	\begin{equation} \nonumber
	\begin{split}
		\dim \Vv(G(\xi))
		  &\le 
		\ \sum_{(k,\bs) \in G(\xi)} \dim V_{2^k} 
		  \le 
		\sum_{\sigma_{2;\bs}^{q_2} \leq 2^{-k}\xi}  2^k \leq \sum_{\bs\in \FF} \xi \sigma_{2;\bs}^{-q_2}= \xi \sum_{\bs\in \FF}  \sigma_{2;\bs}^{-q_2} \leq C\xi.
	\end{split}
\end{equation}
	Hence, for each $n \in \NN$ we can find  a number $\xi_n$  such that  $\dim \Vv(G(\xi_n)) \le  n$ and 
	\begin{equation} \label{u-I_Gu, alpha<}
		\|v -\Ii_{G(\xi_n)}v\|_{\Ll_2(X^1)} \leq Cn^{-\alpha}, \quad \alpha \le 1/q_2 - 1/2.
	\end{equation}
	This proves the result in the case $\alpha \le 1/q_2 - 1/2$.

\noindent Step 5:
	{\it The case $\alpha > 1/q_2  - 1/2$.}
	Lemma \ref{lemma[L_p-approx]} gives 
	\begin{equation} \label{v - Ss_G}
		\big\|v- \Ss_{G (\xi)} v\big\|_{\Ll_2(X^1)} 
		\  \le \  C \xi^{-(1/q_1 - 1/2)}.
	\end{equation}
	
	We split $A(\xi)$ in \eqref{A(xi)} 
        into two sums as
	\begin{equation*} 
		A(\xi)
		= 
		A_1(\xi) + A_2(\xi),
	\end{equation*}
	where
	\begin{equation*} 
		A_1(\xi)
		:= 
		\sum_{\sigma_{1;\bs}^{q_1} > {\xi (\log\xi)^\eta}, \  
			\sigma_{2;\bs} \leq 2^{- \tau k}\xi^\vartheta}  {\|\delta_k(v_\bs) \|_{X^1}}  
		p_\bs(\theta, \lambda),	
	\end{equation*}
	and
	\begin{equation*} 
		A_2(\xi)
		:= 
		\sum_{ 
			\sigma_{2;\bs} > 2^{- \tau k}\xi^\vartheta}  {\|\delta_k(v_\bs) \|_{X^1}} \, 
		p_\bs(\theta, \lambda).
	\end{equation*}
	We get by Assumption~\ref{assum2}, item (iii),
	\begin{equation} \label{A_1<}
		\begin{split}
			A_1(\xi)
			\ &\le \ 	
			\sum_{\sigma_{1;\bs}^{q_1} > 
                            \xi (\log\xi)^\eta} {\sum_{k=0}^{k(\xi)}} \|\delta_k(v_\bs)\|_{X^1}	p_\bs(\theta, \lambda)
			\\
			\  &\le  \
			C \sum_{\sigma_{1;\bs}^{q_1} > 
                             \xi (\log\xi)^\eta} \sum_{k=0}^{k(\xi)}\|v_\bs\|_{X^1} p_\bs(\theta,\lambda)
			\\
			\ &\le \ 	
			C	\log \xi \sum_{\sigma_{1;\bs}^{q_1} > \xi (\log\xi)^\eta}
			\|v_\bs\|_{X^1}	 \, p_\bs(\theta,\lambda).
		\end{split}
	\end{equation}
	We obtain by H\"older's inequality and the hypothesis of the theorem,
	\begin{equation*} 
		\begin{split}
			\sum_{\sigma_{1;\bs}^{q_1} > \xi (\log\xi)^\eta}   &\|v_\bs\|_{X^1}  \, p_\bs(\theta,\lambda)
			\\ &\le \
			\left(\sum_{\sigma_{1;\bs}^{q_1} > \xi (\log\xi)^\eta} (\sigma_{1;\bs}\|v_\bs\|_{X^1})^2\right)^{1/2}
			\left(\sum_{\sigma_{1;\bs}^{q_1} > \xi (\log\xi)^\eta} p_\bs(\theta,\lambda)^{2} \sigma_{1;\bs}^{-2}\right)^{1/2}
			\\[1.5ex]
			\ &\le \
			C\,
			\left(\sum_{\sigma_{1;\bs}^{q_1} > \xi (\log\xi)^\eta} p_\bs(\theta,\lambda)^{2} \sigma_{1;\bs}^{-q_1} 
			\sigma_{1;\bs}^{-(2- q_1)}\right)^{1/2}
			\\[1.5ex]
			\ &\le \
			C \big(\xi(\log \xi)^\eta\big)^{-(1/q_1 - 1/2)}
			\left(\sum_{\bs \in \FF} p_\bs(\theta,\lambda)^{ 2} \sigma_{1;\bs}^{-q_1} \right)^{1/2}
			\\[1.5ex]
			\ &\le \
			C \xi^{-(1/q_1 - 1/2)}
			(\log \xi)^{-\eta(1/q_1 - 1/2)}.
		\end{split}
	\end{equation*} 
	Due to the equality $\eta(1/q_1 - 1/2) =1$, this and \eqref{A_1<} yield that
	\begin{equation} \label{A_1}
		A_1(\xi)
		\  \le \  C \xi^{-(1/q_1 - 1/2)}.
	\end{equation}
	
	We now give a bound for $A_2(\xi)$. 
        Observe that $\vartheta \alpha/\tau = 1/q_1 - 1/2$ and 
	$\alpha/\tau =1 -  q_2/2$. 
	Employing \eqref{del k1}, 
        the assumption~\eqref{weighted summability} 
        and H\"older's inequality, we get 
	\begin{equation*} 
		\begin{split}
			A_2(\xi)
			\ &\le \
			\sum_{\bs \in \FF} \
			\sum_{ 2^k > \left(\xi^\vartheta\sigma_{2;\bs}^{-1}\right)^{1/\tau} }
			\|	\delta_k(v_\bs)\|_{X^1} \, 
			p_\bs(\theta,\lambda)
			\\[1.5ex]
			\ &\le \
			C \, \sum_{\bs \in \FF} \
			\sum_{ 2^k > \left(\xi^\vartheta\sigma_{2;\bs}^{-1}\right)^{1/\tau} }
			2^{-\alpha k} \|v_\bs\|_{X^2}\, p_\bs(\theta,\lambda)							
			\\[1.5ex]
			\ &\le \
			C\,
			\sum_{\bs \in \FF} \
			\big(  \sigma_{2;\bs}^{-1/\tau} \xi^{\vartheta /\tau}\big)^{- \alpha }\|v_\bs\|_{X^2}\, p_\bs(\theta,\lambda)
			\\[1.5ex]
			\ &= \
			C\,
			\xi^{- \vartheta \alpha/\tau }\sum_{\bs \in \FF} \
			\sigma_{2;\bs}^{\alpha/\tau} 
			\|v_\bs\|_{X^2}\, p_\bs(\theta,\lambda)
			\\[1.5ex]
			\ &= \
			C\,
			\xi^{- (1/q_1 - 1/2)}\	\sum_{\bs \in \FF} \
			\sigma_{2;\bs}^{1 - q_2/2} 
			\|v_\bs\|_{X^2}\, p_\bs(\theta,\lambda)
			\\
			\ &\le \
			C\,
			\xi^{- (1/q_1 - 1/2)}\	
			\left(\sum_{\bs \in \FF}  (\sigma_{2;\bs}\|v_\bs\|_{X^2})^2\right)^{1/2}
			\left(\sum_{\bs \in \FF}  p_\bs^2(\theta,\lambda) \sigma_{2;\bs}^{-q_2}\right)^{1/2}
			\\
			\ &\le \
			C\,
			\xi^{- (1/q_1 - 1/2)}.
		\end{split}
	\end{equation*}
	This proves that 
	\begin{equation} \label{A_2}
		A_2(\xi)
		\  \le \  C \xi^{-(1/q_1 - 1/2)}.
	\end{equation} 
	
	Combining \eqref{A_1}, \eqref{A_2}  and \eqref{v - Ss_G} leads to the estimate
	\begin{equation} \label{xi-rate4}
		\|v- \Ii_{G(\xi)} v\|_{\Ll_2(X^1)}
		\  \le \  C \xi^{-(1/q_1 - 1/2)}.
	\end{equation}
	
	We   estimate the dimension of the space $\Vv(G(\xi))$.  
        Put $q:= \tau q_2 $ and define $q'$ by $1/q' + 1/q = 1$. 
	Since $\alpha > 1/q_2 - 1/2$, we have $q >1$. 
        Consequently, 
        using H\"older's inequality and \eqref{weighted summability}, 
        we derive that
	\begin{equation} \nonumber
		\begin{split}
			\dim \Vv(G(\xi))
			\ &\le \
			\ \sum_{(k,\bs) \in G(\xi)} \dim V_{2^k} 
			\ \le \ 
			\sum_{\sigma_{1;\bs}^{q_1} \le \xi(\log\xi)^\eta} \, 
			\sum_{2^{\tau k} \sigma_{2;\bs} \leq \xi^\vartheta} 2^k 
			\\ \  &\le \
			2 \sum_{\sigma_{1;\bs}^{q_1} \le \xi(\log\xi)^\eta} \xi^{\vartheta/\tau}\sigma_{2;\bs}^{-1/\tau} 
			\\& \le \
			2 \xi^{\vartheta/\tau} \left(\sum_{\sigma_{1;\bs}^{q_1} \le \xi(\log\xi)^\eta}\sigma_{2;\bs}^{-q_2} \right)^{1/q}
			\left(\sum_{\sigma_{1;\bs}^{q_1} \le \xi(\log\xi)^\eta} 1 \right)^{1/q'}
			\\ \  &\le 
			2 \xi^{\vartheta/\tau}\left(\sum_{\bs \in \FF} \sigma_{2;\bs}^{-q_2} \right)^{1/q}
			\left(\sum_{\bs \in \FF} \xi (\log\xi)^\eta \sigma_{1;\bs}^{-q_1}  \right)^{1/q'}
			\\ \  &=
			M \xi^{\vartheta/\tau + 1/q'} (\log\xi)^{\eta/q'}
			\ =  \
			M \xi^{1 + \delta/\alpha}(\log\xi)^{\eta/q'},
		\end{split}
	\end{equation}
	where
	$
	M:=  2 \big\|\big(\sigma_{2;\bs}^{-1}\big)\big\|_{\ell_{q_2}(\FF)}^{q_2/q} \,
	\big\|\big(\sigma_{1;\bs}^{-1}\big)\big\|_{\ell_{q_1}(\FF)}^{q_1/q'}.
	$
	
	For any $n \in \NN$, letting $\xi_n$ be a number satisfying the inequalities 
	\begin{equation} \label{[xi_n]}
		M\, \xi_n^{1 + \delta/\alpha}(\log\xi_n)^{\eta/q'}
		\ \le \
		n
		\ < \ 2M\, \xi_n^{1 + \delta/\alpha}(\log\xi_n)^{\eta/q'},
	\end{equation}
	we derive that  $\dim \Vv(G(\xi_n)) \le  n$.
	On the other hand, by \eqref{[xi_n]},
	\begin{equation} \nonumber
		\xi_n^{-(1/q_1 - 1/2)} 
		\asymp \bigg(\frac{n}{(\log n)^{\eta/q'}}\bigg)^{-(1/q_1 - 1/2)\frac{\alpha}{\alpha + \delta}}
		\ = \ n^{-\beta}(\log n)^\kappa,
	\end{equation}
	where with $q_1 < 2$ and $\alpha > 1/q_2 - 1/2$, we have
	\begin{equation} \nonumber
		0 \ < \	\kappa := \
		\frac{\alpha +1/2 - 1/q_2}{\alpha +1/q_1 - 1/q_2} \ < \ 1.
	\end{equation}
	This together with \eqref{xi-rate4} proves that
	\begin{equation} \nonumber
		\|v -\Ii_{G(\xi_n)}v\|_{\Ll_2(X^1)} \leq Cn^{-\beta}(\log n)^\kappa, \quad \alpha > 1/q_2 - 1/2.
	\end{equation}
	
	By combining the last estimate and \eqref{u-I_Gu, alpha<} we derive  \eqref{vSGxi3}.
	\hfill
\end{proof}
\begin{remark} 
	\rm{
		Fully discrete GPC interpolation approximation of functions in 
		Bochner spaces with countable product tensor product Jacobi measure 
                and applications to affine-parametric PDEs  
                was studied in 
		\cite{Dung19,Dung21,Zec18T,ZDS19}. 
                Theorem~\ref{thm: FullyPolAppr} in the case 
		$\alpha \le \frac{1}{q_2}-\frac{1}{2}$ was proven in  \cite[Theorem 6.3]{Dung21-v11}.  
	Theorem~\ref{thm: FullyPolAppr} in the case 
	$\alpha > \frac{1}{q_2}-\frac{1}{2}$ improves the result  of  \cite[Theorem 6.3]{Dung21-v11} by a logarithm factor of $(\log n)^{-(1-\kappa)}$.  }
\end{remark}

\begin{definition}\label{def:G_rev(xi)}
	For $\xi >1$, $\tau$, $\vartheta$ and $\eta$ as in \eqref{vartheta:=},
	denote
	\begin{equation*}\label{G_rev(xi)}
		G_\rev(\xi)
		:=
		\begin{cases}
			\sett{
				(k,\bs)\in \NN_0\times \FF_\rev: 2^k \sigma_{2;\bs}^{q_2} \le \xi
			}
			&
			\text{if } \alpha \le 1/q_2 - 1/2,
			\\
			\sett{
				(k,\bs)\in \NN_0\times \FF_\rev: \sigma_{1;\bs}^{q_1}\le \xi (\log \xi)^\eta, \ 
				2^{\tau k} \sigma_{2;\bs}  \le \xi^\vartheta
                        }
			&
			\text{if } \alpha > 1/q_2 - 1/2.
		\end{cases}
	\end{equation*}
\end{definition}
In a similar way as in the proof of Theorem~\ref{thm: FullyPolAppr},
we can prove 
\begin{corollary}\label{corollary: FullyPolApprEv}
	Let $v \in \Ll_2(X^2)$ admit a even Jacobi GPC expansion
	\begin{equation} \label{JacobiSeriesVev}
		v = \sum_{\bs\in\FF_\rev} v_\bs J_\bs, \quad v_\bs \in X^2.
	\end{equation}
	Suppose that $v$ satisfies Assumption~\ref{assum2} with $\FF$ being replaced by $\FF_\rev$
	and let $G_\rev(\xi)$ be as in Definition \ref{def:G_rev(xi)}.
	
	Then for each $n\in \NN$ there exists a number $\xi_n$ such that 
	for the symmetric sparse tensor interpolation operator 
	$$
	\Ii^*_{G_\rev(\xi_n)}: C(\IIi, X^2) \to \Vv(G_\rev(\xi_n)),
	$$ 
	holds $\dim \Vv(G_\rev(\xi_n)) \le  n$.
	
	Furthermore, 
	there exists a constant $C>0$ such that for $n\in \NN$ and for $0 < p \le 2$ 
	it holds that
	\begin{align*}
		\norm{v-\Ii_{G_\rev(\xi_n)}v}{\Ll_p( X^1)} 
		\le
		C
	\begin{cases}
			n^{-\alpha} &\text{if } \alpha \leq  1/q_2 - 1/2,
			\\
			n^{-\beta}(\log n)^\kappa &\text{if } \alpha >  1/q_2 - 1/2,
		\end{cases} 
		\label{vSGxi3Rev}
	\end{align*}
	where the convergence rate $\alpha$ is  given by~\eqref{SpatialAppr}, 
        and $\beta$ and $\kappa$ by \eqref{beta}.	
\end{corollary}

\subsection{Sparse-grid quadratures}
\label{Integration}
In this section, we construct linear fully-discrete quadratures for numerical integration
of functions taking values in $X^2$ 
with double-weighted $\ell_2$-summability of Jacobi GPC expansion coefficients  
for Hilbert spaces $X^1$ and $X^2$ satisfying a certain ``spatial" approximation property, 
as specified in Assumption~\ref{assum2}, and their bounded linear functionals. 
In particular,
we give convergence rates for these quadratures
which are derived from the results on convergence rate of 
polynomial interpolation approximation in  
$\Ll_1(X^1)$ in Corollary~\ref{corollary: FullyPolApprEv}.

Assume that $a_j = b_j$ for $j \in \NN$ in the definition \eqref{eq:ProdJPM} of the measure $\mu_{\ba,\bb}$. 
This case corresponds to the symmetric ultra-spherical measure $\mu_{\ba,\ba}$. 
This symmetry property allows to establish in the application settings in the next section,  
a crucial improvement of the convergence rate of sparse-grid quadrature
with respect to the infinite-tensor-product, ultra-spherical weight 
due to the cancellation of anti-symmetric terms.

If $v$ is a function defined on $\II$ taking values in a Hilbert space $X$,  
the function  $I_m(v)$  defined in \eqref{Imv Lmk}  
generates the $j$th component interpolatory quadrature formulas 
defined as
\begin{equation} \nonumber
	Q_{\nu_j}(v)
	:= \ \int_{\II} I_{\nu_j}(v)(y_j) \, \rd\mu_{a_j,b_j}(y_j)
	\ = \
	\sum_{k=0}^{\nu_j}\omega_{\nu_j,k}\, v(y_{\nu_j,k}),  \ \ j \in \NN,
\end{equation}	
where the quadrature weights are given by
\begin{equation} \nonumber
	\omega_{\nu_j,k} :=  \int_{\II} L_{\nu_j,k}(y_j) \, \rd\mu_{a_j,b_j}(y_j).
\end{equation}
The  quadrature $Q_{\nu_j}$ being interpolatory, 
we have for every polynomial $\varphi$ in variable $y_j$ of degree $\le \nu_j$,
\begin{equation} \nonumber
	Q_{\nu_j}(\varphi)
	\  = \ 
	\int_{\II} \varphi(y_j) \, \rd\mu_{a_j,b_j}(y_j)
	\;.
\end{equation}	
We define the univariate ``increment'' or ``detail'' operator 
$\Delta^{{\rm Q}}_{\nu_j}$ for even $\nu_j \in 2 \NN_0$ 
by
\begin{equation} \nonumber
	\Delta^{{\rm Q}}_{\nu_j}
	:= \
	Q_{\nu_j} - Q_{\nu_j - 2},
	\quad \mbox{with the convention $Q_{-2} := 0$}.
\end{equation} 
For a function $v\in C(\IIi;X)$, 
for $\bs \in \FF_\rev$ introduce the tensorized increment operator 
(with respect to the measure $\mu=\mu_{\ba,\bb}$)
\begin{equation} \nonumber
	\Delta^{{\rm Q}}_\bs(v)
	:= \
	\bigotimes_{j \in \NN} \Delta^{{\rm Q}}_{\nu_j}(v),
\end{equation} 
where the univariate operator
$\Delta^{{\rm Q}}_{\nu_j}$ 
is applied to the univariate function 
$\bigotimes_{i < j} \Delta^{{\rm Q}}_{\nu_i}(v)$ 
by considering $v$ as a function of variable $y_j$ 
with all remaining variables held fixed. 
As $\Delta^{{\rm I}}_\bs$, 
the operators  $\Delta^{{\rm Q}}_\bs$ are well-defined for all $\bs \in \FF_\rev$. 

Let Assumption~\ref{assum2}  hold for Hilbert spaces $X^1$ and $X^2$, 
and $v \in \Ll_2(X^2)$. 
For a finite index set $G \subset \NN_0 \times \FF_\rev$ 
with the structure \eqref{G(xi)},
we introduce the quadrature operator $\Qq_G$ 
which is generated by the 
\emph{symmetric Smolyak sparse-grid tensor-product interpolator}
$\Ii^*_G:  C(\IIi, X^2) \to \Vv(G)$, 
defined for $v \in C(\IIi, X^2)$ by
\begin{equation*}  \label{Qq=int}
	\Qq_G v
	:= \
	\sum_{(k,\bs) \in G} (\delta_k \otimes \Delta^{{\rm Q}}_\bs) (v)
	\ = \
	\int_{\IIi} \Ii^*_G v (\by)\, \rd\mu(\by).
\end{equation*}
Further, if $\phi \in (X^1)'$  is a bounded linear functional on $X^1$, 
for a finite index set 
$G \subset \NN_0 \times \FF_\rev$, 
with the structure \eqref{G(xi)},
the quadrature formula $\Qq_G v$ generates the quadrature formula 
$\Qq_G \langle \phi, v \rangle$ for integration of $\langle \phi, v \rangle$ 
by
\begin{equation} \nonumber
	\Qq_G \langle \phi, v \rangle
	:= \
	\langle \phi, \Qq_G v  \rangle
	\ = \
	\int_{\IIi} \langle \phi, \Ii^*_G v (\by) \rangle\, \rd\mu(\by).
\end{equation}

The sets $Y_m$ in \eqref{eq:ChebNod} 
of Chebyshev nodes are symmetric with respect to the origin for every $m \in \NN_0$.
In the ultra-spherical case $a_j=b_j$, $j \in \NN$
of the Jacobi probability measure $\mu(\by)$,
then it holds $\int_{\II} J_k(x) \rd x =0$, if $k \in \NN$ is odd.
For a function  $v \in \Ll_2(X)$ that is represented by the Jacobi gpc series \eqref{J-SeriesX}, 
the assumed symmetry of the product measure $\mu$ with respect to $\bsnul$ 
implies (see \eqref{eq:JacSym})
\begin{equation*} \label{int v = int v_ev}
	\int_{\IIi} v(\by) \, \rd\mu(\by)
	\ = \
	\int_{\IIi} v_\rev(\by) \, \rd\mu(\by),
\end{equation*}
and
\begin{equation*} \label{Delta^QH_s=0}
	\Delta^{{\rm Q}}_{\bs'} J_\bs(\by) 
	\ = \
	0, \quad \bs \notin \FF_\rev, \ \bs' \in \FF.
\end{equation*}

\begin{theorem}\label{thm[quadrature]} 
Let  $a_j=b_j$, $j \in \NN$, for the Jacobi probability measure $\mu_{\ba,\bb}$.
Consider a function  $v \in \Ll_2(X^2)$ represented by the
Jacobi GPC series  \eqref{JacobiSeriesVev}, 
and such that $v$ satisfies Assumption~\ref{assum2} 
with $\FF$ being replaced by $\FF_\rev$.
For $\xi > 0$, 
let $G_\rev(\xi)$ be the index set in Definition~\ref{def:G_rev(xi)}.

Then, 
there is a constant $C>0$ (depending on $v$) 
such that
	\begin{itemize}
		\item[{\rm (i)}]
		For each $n \in \NN$,
                there exists a number $\xi_n$  
                such that  $\dim\Vv(G_\rev(\xi_n))\le n$ 
                and 
		\begin{equation*} \label{u-Q_Gu}
		\left\|\int_{\IIi}v(\by)\, \rd\mu(\by) - \Qq_{G_\rev(\xi_n)}v\right\|_{X^1} 
                \leq 
        	C
	\begin{cases}
		n^{-\alpha} &\text{if } \alpha \leq  1/q_2 - 1/2,
		\\
		n^{-\beta}(\log n)^\kappa &\text{if } \alpha >  1/q_2 - 1/2.
	\end{cases} 
		\end{equation*}
		\item[{\rm (ii)}] 
		Let  $\phi \in (X^1)'$ be a bounded linear functional on $X^1$. 
		Then, for each $n \in \NN$ 
                there exists $\xi_n \in \RR$ such that  $\dim\Vv(G_\rev(\xi_n))\le n$ 
		and 
		\begin{equation*} \label{u-Q_Gu_phi}
			\left|\int_{\IIi} \langle \phi,  v (\by) \rangle\,  \rd\mu(\by) 
                               - \Qq_{G_\rev(\xi_n)} \langle \phi,  v \rangle\right| 
                \leq C\|\phi\|_{(X^1)'}
		\begin{cases}
			n^{-\alpha} &\text{if } \ \alpha \leq  1/q_2 - 1/2,
			\\
			n^{-\beta} (\log n)^\kappa &\text{if } \ \alpha >  1/q_2 - 1/2.
		\end{cases} 
		\end{equation*}
	\end{itemize}	
	The rate $\alpha$ is given by~\eqref{SpatialAppr}, $\beta$ and $\kappa$ by \eqref{beta}.
\end{theorem}

\begin{proof}	
Using Corollary \ref{corollary: FullyPolApprEv}, 
this theorem can be proven in a similar way to the proof of \cite[Theorem 4.1]{Dung23Err} 
(see also \cite[Theorem 4.1]{Dung21-v11}) with some necessary modifications  
as in the proof of Theorem \ref{thm: FullyPolAppr}.
\hfill
\end{proof}
\section{Applications}
\label{sec:Expl}
We indicate practical applications of the preceding results.
The first, illustrative of these is the 
affine-parametric, linear elliptic model equation
which was considered in many other references, e.g.
in 
\cite{CDS10,CDS11,CCDS13,BCDC17,BCM17,DGHR18,ZDS19,MBVoulis22,Dung19,Dung21,Dung21-v11} 
and references there.

The second of these is the more abstract setting of holomorphic, implicit
operator equations as in e.g. \cite{CSZ16,MaxwShHol17,MaxwDomUQ20},
with affine-parametric input encoding as considered, e.g. \cite{CCS15,Zec18T}. 
\subsection{Affine-parametric, linear elliptic PDEs}
\label{sec:AffParEllPDE}

Let $D \subset \RR^d$ be a bounded  Lipschitz domain. 
Consider the linear diffusion elliptic equation in the divergence form
\begin{equation*} \label{ellip-0}
	- \dv (a\nabla u)
	\ = \
	f \quad \text{in} \quad D,
	\quad u|_{\partial D} \ = \ 0, 
\end{equation*}
for  a given fixed right-hand side $f$ and a 
spatially variable scalar diffusion coefficient $a$.
Denote by $V:= H^1_0(D)$ the energy space and $V^*:=H^{-1}(D)$ the dual space of $V$. 
Assume that  $f \in V^*$ (in what follows this preliminary assumption always holds without mention). 
If $a \in L_\infty(D)$ satisfies the ellipticity assumption
\begin{equation} \nonumber
0<\underset{\bx\in D}{\operatorname{ ess inf}}\,a(\bx) \leq a(\bx) \leq  \underset{\bx\in D}{\operatorname{ ess sup}}\, a(\bx)<\infty,\ \ \bx\in D,
\end{equation}
by the well-known Lax-Milgram lemma, there exists a unique 
solution $u \in V$  to the equation~\eqref{ellip}  in the weak form
\begin{equation} \nonumber
	\int_{D} a(\bx)\nabla u(\bx) \cdot \nabla v(\bx) \, \rd \bx
	\ = \
	\langle f , v \rangle   \quad \forall v \in V.
\end{equation}
Let us  consider the parametric diffusion elliptic equation
\begin{equation} \label{ellip}
	- \dv (a(\by)\nabla u(\by))
	\ = \
	f \quad \text{in} \quad D,
	\quad u|_{\partial D} \ = \ 0  \ \ \by \in \IIi,
\end{equation}
with affine-parametric diffusion coefficients 
\begin{equation*} \label{affine}
	a(\by)= \bar a + \sum_{j = 1}^\infty y_j\psi_j,\qquad \by\in \II^\infty,
\end{equation*}
where $\bar{a}$ and $(\psi_j)_{j\in \NN}$ belong to $L_\infty(D)$. 
Note that if
$$
a_{\min}:=\inf_{\by\in \II^\infty} \inf_{\bx\in D} a(\by) (\bx)>0,
$$
then 
\begin{equation*}\label{eq:uniformBounded}
\sup_{\by \in \IIi}	\|u(\by)\|_V\leq \frac{\|f\|_{V^*}}{a_{\min}}.
\end{equation*}

We introduce the space
$W^r:=\{v\in V  :   \Delta v\in H^{r-2}(D)\}$
for $r\ge 2$. 
This space is equipped with the norm
$\|v\|_{W^r}:=\|\Delta v\|_{H^{r-2}(D)}$,
and coincides with the Sobolev space $V\cap H^{r-2}(D)$ 
with equivalent norms if the domain $D$ has $C^{r-1,1}$ smoothness,
see, e.g., \cite[Theorem 2.5.1.1]{Gri85B}. 
We make use of the convention $W^1:= V$.
\begin{lemma} \label{lemma[summability]J}	
	Let $r \in \NN$. Let the right side $f$ in \eqref{ellip}
	belong to $H^{r-2}(D)$, and $D$  be a bounded Lipschitz domain for $r=1$ and a bounded domain of 
	 $C^{r-2,1}$ smoothness for $r\ge 2$. 
        Assume that $\bar a\in L_\infty(D)$  
        is such that ${\rm ess} \inf \bar a>0$,
	and that there exist sequences 
        $\brho_1=(\rho_{1;j}) _{j \in \NN}$ and $\brho_r=(\rho_{r;j}) _{j \in \NN}$
        of positive numbers such that 
	\begin{equation*} \label{eq-affine-condition}
		\left \| \frac{\sum _{j \in \NN} \rho_{1;j}|\psi_j|}{\bar a} \right \|_{L_\infty(D)} 
		< 1 \;,
	\end{equation*} 	
	and, that in the case $r \ge 2$, $\bar a$ and all functions $\psi_j$ belong to $W^{r-1,\infty}(D)$,
        that there holds
	\begin{equation} \label{eq-rho-r}
		\sup_{|\alpha|\leq r-1} 
		\Bigg\| \sum _{j \in \NN} \rho_{r;j} |D^\alpha \psi_j| \Bigg\|_{L_\infty(D)} 
		<\infty\;.
	\end{equation}	
	Then for the sequence $\bsigma_r = \brac{\sigma_{r;\bnu}}_{\bnu \in \FF}$,
	\begin{equation} \label{beta_r,s}
		\sum_{\bs\in\FF} \big(\sigma_{r;\bs} \|u_\bs\|_{W^r}\big)^2 < \infty, \quad 
		\sigma_{r;\bs} := \brho_r^\bs\prod _{j \in \NN} c_{\nu_j}^{{a_j,b_j}},
	\end{equation}	
where we recall the constants $c_{k}^{a,b}$, $k\in \NN_0$, as defined in \eqref{eq-cabk}.
\end{lemma}

\begin{proof}
This lemma has been proven in \cite[Theorem 3.1]{BCM17} for $r=1$ and the Legendre expansion, 
and 
in \cite[Theorem 5.1]{BCDC17} for $r > 1$ and the scalar case when $a_j =a$, $b_j =b$ for $j \in \NN$. 
The general case can be proven in an analogous fashion with certain modifications.
\hfill	
\end{proof}



\begin{lemma} \label{lemma:bcdmJ}
	Let $0 < q <\infty$, $\kappa \in \NN$ and $\brho=(\rho_j)_{j \in \NN}$ be a sequence 
of  numbers larger than $1$ such the sequence $(\rho_j^{-1}) _{j \in \NN}$ belongs to $\ell_q(\NN)$. 
Let $(p_\bs(\theta,\lambda))_{\bs \in \FF}$ be the sequence of the form  \eqref{[p_s]} 
with arbitrary nonnegative $\theta, \lambda$. 

Then for  the sequence $\bsigma=(\sigma_\bs)_{\bs \in \FF}$ defined by
	\begin{equation} \label{sigma_bs}
		\sigma_\bs  :=  \brho^\bs \prod_{j \in \NN} c_{\nu_j}^{a_j,b_j},
	\end{equation}
	we have
	\begin{equation} \nonumber
		\sum_{\bs \in \FF_\kappa} 
		p_\bs(\theta,\lambda) \sigma_\bs^{-q/\kappa}< \infty.
	\end{equation}
\end{lemma}
	\begin{proof}
        There holds $c_k^{a,b} \le  (1 + k)^{\tau}$ for  $k \in \NN_0$
		with some   $\tau > 0$ depending on $a,b$. 
        Since $\ba,\bb\in \ell_\infty(\NN)$ we have for $\bnu\in \FF$ that
		$$
\prod_{j\in \NN}	c_{\nu_j}^{a_j,b_j}	\leq \prod_{j\in \NN}(1+\nu_j)^{\theta'}
		$$
	with $\theta'>0$ depending on $|\ba|_\infty$ and $|\bb|_\infty$.
                For any $\theta \ge 0$, 
		we get with \eqref{sigma_bs}
		\begin{align*} \label{}
		p_\bs(\theta,\lambda) \sigma_\bs^{-q/\kappa} 
                & = p_\bs(\theta,\lambda)  ( \brho^{-\bs})^{q/\kappa} 
                                \Bigg(\prod_{j \in \NN} c_{\nu_j}^{a_j,b_j}\Bigg)^{-q/\kappa}
			\\
			&
			\le  p_\bs(\theta,\lambda)  p_\bs(q \theta'/\kappa,1)( \brho^{-\bs})^{q/\kappa}
			\le p_\bs(\theta^*,\lambda^*) ( \brho^{-\bs})^{q/\kappa},
		\end{align*}
		where 
		$\theta^*:= \theta+q\theta'/\kappa$ and $\lambda^*:=\lambda+1$.
		We derive that
		\begin{equation*} \label{}
			\sum_{\bs \in \FF_\kappa} p_\bs(\theta,\lambda) \sigma_\bs^{-q/\kappa}  
			\le 
			\sum_{\bs \in \FF_\kappa} 
			p_\bs(\theta^*,\lambda^*) ( \brho^{-\bs})^{q/\kappa}. 
		\end{equation*}
		Now applying \cite[Lemma 6.2]{Dung21} to the right-hand side we obtain the desired result.
		\hfill	
	\end{proof}

For finite-parametric approximations it is 
necessary to replace the GPC coefficients $u_\bs$ 
by corresponding finite-parametric surrogates. 
For parametric PDEs such as \eqref{ellip}, such 
approximations are furnished by discretizations
which realize the projectors $P_m$ in Assumption~\ref{assum2}, item (iii).

In fully discrete approximations of the solution $u(\by)$ to the parametric 
PDE  \eqref{ellip} by using interpolation with respect to the parametric variables with  a 
large, finite number of  particular values $u(\by_j)$, 
stable numerical approximations depend on discretization of $u(\by_j)$. 
For \emph{uniformly (w.r. to the parameter $\by$) stable PDEs and corresponding uniformly
stable discretizations}
the discretization error is known to be quasi-optimal, uniformly with respect to $\by$; 
convergence rate bounds can in this case be obtained
by replacing the 
discretization error $u(\by_j) - P_m(\by_j) u(\by_j)$ 
by
$u(\by_j) - P_m u(\by_j)$ 
with a suitable (quasi-)interpolant $P_m$ 
as stipulated, e.g.,  in Assumption~\ref{assum2}, item (iii).

\begin{assumption}\label{assum4}
There are
	\begin{itemize}
		\item[{\rm (i)}] 
		 a sequence $(V_m)_{m \in \NN_0}$ of subspaces $V_m \subset V $ of dimension $\le m$, 
		and 		
		\item[{\rm (ii) }]  
		a sequence $(P_m)_{m \in \NN_0}$ of linear operators from $V$ into $V_m$, 
		and a number $\alpha>0$ such that 
		there are stability and consistency constants $C_1,C_2>0$ with
		\begin{equation} \label{spatialappnJ}
			\|P_m(w)\|_V \leq 
			C_1 \|w\|_V , \quad
			\|w-P_m(w)\|_V \leq 
			C_2 m^{-\alpha} \|w\|_{W^r}, \quad \forall m \in \NN_0, \ \  \forall w \in W^r.
		\end{equation}
	\end{itemize}
\end{assumption}

To treat fully discrete approximations,
we assume that $f \in H^{r-2}(D)$ in the equation \eqref{ellip} 
and that it holds the  approximation property \eqref{spatialappnJ} 
in Assumption~\ref{assum4}
for all $w \in W^r$, see, for instance, \cite[Theorem 3.2.1]{Cia78} 
for the case when $D\subset \RR^2$ is a polygonal domain.
Notice  that classical error estimates yield
the convergence rate $\alpha=(r-1)/d$ by using Lagrange finite elements of order at least $r-1$
on quasi-uniform partitions with the finite element spaces $V_m$ associated to grids $(\Tt_m)_{m \in \NN}$ and finite element functions $P_m(w)$ (for $m = 1/h\in \NN$). 
Also, the spaces $W^r$ do not always coincide with $H^r(D)$. 
For example, for $d=2$, $W^r$ is strictly larger than $(H^r\cap H^1_0)(D)$
when $D$ is a polygon with re-entrant corner. 
In this case, it is well known that the optimal rate $\alpha=(r-1)/2$ is yet attained, 
by using the finite element grids $(\Tt_m)_{m\in \NN}$ 
with proper refinement near re-entrant corners where 
$w\in W^r$ might have singularities.

As before, for $w\in V$ and for each $k\in \NN_0$, we define 
\begin{align*}
\delta_0(w):= P_0(w), \;\;
	\delta_k(w)
	&
	:=
	P_{2^k}(w) -  P_{2^{k-1}}(w),
	\quad 
	k\in \NN.
\end{align*}

\begin{theorem}\label{thm[coll-approx]pdeJ}
	Let $0 < p \le 2$ and $r\in \NN$, $r>1$. 
	Let  Assumption~\ref{assum4} hold.  
        Let the assumptions of Lemma~\ref{lemma[summability]J} hold 
        for the spaces $W^1=V$ and  $W^r$ with $\rho_{1;j}$ strictly larger than 1 for all $j\in \NN$ 
        and $(\rho_{i;j}^{-1})_{j\in \NN}\in \ell_{q_i}(\NN)$ for $i=1,r$, 
        and $0 < q_1\leq q_r <\infty$ and $q_1 < 2$. 
For $\xi > 1$, 
let $G(\xi)$ be the thresholded multi-index set 
in Definition~\ref{def:G(xi)} 
for $\bsigma_1 $ and $\bsigma_r$ in place of $\bsigma_2$ 
as in \eqref{beta_r,s} and $q_2$ in place of $q_r$.
	
Then, for each $n\in \NN$ there exists a number $\xi_n$ such that $\dim \Vv(G(\xi_n)) \le  n$.
	Furthermore, 
	there exists a constant $C>0$ such that 
	for any $0< p \le 2$ and any $n\in \NN$,
	we have 
	for the sparse-grid tensor-product interpolation operator 
	$$
	\Ii_{G(\xi_n)}: C(\IIi, W^r) \to \Vv(G(\xi_n)),
	$$ 
	the error bound for the solution $u(\by)$ to the equation \eqref{ellip}
	\begin{align}
		\norm{u-\Ii_{G(\xi_n)}u}{\Ll_p( V)} 
		\le 
		C
		\begin{cases}
			n^{-\alpha} &\text{if } \alpha \leq  1/q_r - 1/2,
			\\
			n^{-\beta}(\log n)^\kappa &\text{if } \alpha >  1/q_r - 1/2,
		\end{cases} 
		\label{vSGxi3Rev2}
	\end{align}
	where $\alpha$ is the convergence rate given by~\eqref{SpatialAppr}, 
        $\beta$ and $\kappa$ by \eqref{beta} with $q_2$ being replaced by $q_r$.
\end{theorem}

\begin{proof}
	First note that by the condition $(\rho_{r;j}^{-1})_{j\in \NN}\in \ell_{q_r}(\NN)$ we have $\rho_{r;j}\to \infty$ as $j\to \infty$. Therefore with out loss of generality we can assume that the sequence $\brho_r$ in \eqref{eq-rho-r} satisfies $\rho_{r;j}>1$ for all $j\in \NN$. 
Consequently, by Lemmas \ref{lemma[summability]J} and \ref{lemma:bcdmJ},
we conclude that Assumption~\ref{assum2} 
holds with $X^1=V$ and $X^2=W^r$ 
for the solution $u$ of the equation \eqref{ellip} 
with $q_2$ being replaced by $q_r$ and any $\tau,\lambda> 0$. 
Now by applying Theorem \ref{thm: FullyPolAppr} the result follows. 
\hfill
\end{proof}

\begin{definition} \label{def:tG_rev(xi)}
For given $\brac{\sigma_{i;\bs}}$, $i=1,2$, and threshold parameter $\xi>1$, 
we define the threshold even multi-index set $\tilde{G}_\rev(\xi)$  
by 
\begin{equation*} 
	\tilde{G}_\rev(\xi)
	:= \ 
	\begin{cases}
		\big\{(k,\bs) \in \NN_0 \times\FF_\rev: \, 2^k \sigma_{2;\bs}^{q_2/2} \leq \xi\big\}  &\text{if }  \ \alpha \le 2/q_2 - 1/2,
                 \\
		\big\{(k,\bs) \in \NN_0 \times\FF_\rev: \, \sigma_{1;\bs}^{q_1/2} \le \xi {(\log \xi)^\eta}, \  
		2^{\tau k} \sigma_{2;\bs}\leq \xi^\vartheta \big\}    & \text{if }  \ \alpha > 2/q_2 - 1/2,
	\end{cases}
\end{equation*}
where 
\begin{equation*} \label{tau,vartheta}
	\tau:= \frac{4\alpha}{4 - q_2}, \qquad		
	\vartheta:= \frac{4}{4 - q_2} \left(\frac{2}{q_1} - \frac{1}{2}\right),\  \  \  \eta:=\bigg(\frac{2}{q_1}-\frac{1}{2}\bigg)^{-1}.
\end{equation*}	
\end{definition}

\begin{theorem}\label{thm[quadrature]pdeJ}
Let Assumption~\ref{assum4} hold and  $r\in \NN, \ r>1$.  
Let  $a_j=b_j$, $j \in \NN$, for the Jacobi probability measure $\mu_{\ba,\bb}$, 
and the assumptions of Lemma~\ref{lemma[summability]J} hold  for the spaces $W^1=V$ and  $W^r$
with $\rho_{1;j}$ strictly larger than $1$ for all $j\in \NN$ and 
$(\rho_{i;j}^{-1})_{j\in \NN}\in \ell_{q_i}(\NN)$ for $i=1,r$, 
and $0 < q_1\leq q_r <\infty$ and  $q_1 < 4$. 
For $\xi > 1$, 
let $\tilde{G}_\rev(\xi)$ be the threshold index set 
in Definition~\ref{def:tG_rev(xi)} for $\bsigma_1$ 
and 
$\bsigma_2$  with $\bsigma_2$ being replaced by $\bsigma_r$ as in \eqref{beta_r,s} 
and 
$q_2$ by $q_r$.

Then, for the quadrature operator $\Qq_{\tilde{G}_\rev(\xi)}$ 
generated by the interpolation operator 
$\Ii^*_{\tilde{G}_\rev(\xi)}: C(\IIi, W^r)\to \Vv(\tilde{G}_\rev(\xi))$, 
we have the following for the solution $u(\by)$ to the equation \eqref{ellip}.
\begin{itemize}
	\item[{\rm (i)}]
	There exists a constant $C>0$ such that 
	for any $n\in \NN$ there exists a number $\xi_n$ such that
	$\dim\Vv(\tilde{G}_\rev(\xi_n))\le n$ and 
	\begin{equation} \label{u-Q_Gu-quadratureJ}
		\left\|\int_{\IIi}u(\by)\, \rd \mu(\by) - \Qq_{\tilde{G}_\rev(\xi_n)}u\right\|_V \leq 	C
		\begin{cases}
			n^{-\alpha} &\text{if } \alpha \leq  2/{q_r} - 1/2,
			\\
			n^{-\beta}(\log n)^\kappa &\text{if } \alpha >  2/{q_r} - 1/2.
		\end{cases} 
	\end{equation}			
	\item[{\rm (ii)}] 
	Let $\phi \in V^* $. 
	There exists a constant $C>0$ such that 
	for any $n\in \NN$ there exists a number $\xi_n$ such that $\dim\Vv(\tilde{G}_\rev(\xi_n))\le n$ and 
	\begin{equation*} \label{u-Q_Gu_phiJ}
		\left|\int_{\IIi} \langle \phi,  u(\by) \rangle \rd \mu(\by) - \Qq_{\tilde{G}_\rev(\xi_n)} \langle \phi,  u \rangle\right| \leq C \|\phi\|_{V^*}
		\begin{cases}
			n^{-\alpha} &\text{if } \alpha \leq  2/{q_r} - 1/2,
			\\
			n^{-\beta}(\log n)^\kappa &\text{if } \alpha >  2/{q_r} - 1/2.
		\end{cases}
	\end{equation*}
\end{itemize}	
	The rate $\alpha$ is given by~\eqref{SpatialAppr} and $\beta >0$ is given by
	\begin{equation} \label{eq:beta2}
		\beta := \left(\frac 2 {q_1} - \frac{1}{2}\right)\frac{\alpha}{\alpha + \delta}, \ 
                         \kappa=\frac{\alpha+1/2-2/q_r}{\alpha+2/q_1-2/q_r} \quad 
                \mbox{with}\quad 
		\delta := \frac{2}{q_1}-\frac{2}{{q_r}}.
	\end{equation}
\end{theorem}

\begin{proof}
Observe that $\FF_\rev \subset \FF_2$.
From Lemma~\ref {lemma[summability]J} and Lemma~\ref{lemma:bcdmJ},
the assumptions of Theorem~\ref{thm[quadrature]}  hold for $X^1=V$
and $X^2 = W^r$ with $0 < q_1/2 \le q_r/2 < \infty$ and $q_1/2 < 2$. 
Theorem~\ref{thm[quadrature]pdeJ} follows by applying Theorem~\ref{thm[quadrature]}, 
with $q_1/2$ in place of $q_1$ and $q_r/2$ in place of $q_2$.
\hfill
\end{proof}
%
%
\subsection{Holomorphic maps with affine-parametric encoding}
\label{sec:IOPAff}
We consider abstract, 
real analytic maps $\frku$ between Hilbert spaces $Z$ and $X$.
We identify the real Hilbert spaces $X,Z$ with their complexification
$X_\CC, Z_\CC$,
without change in notation. 
Real analytic maps $\frku:Z\to X$ admit
unique holomorphic extensions, again denoted by $\frku$,
to the complexifications $X_\CC, Z_\CC$ by a power series argument.

Assume given a sequence $(\psi_j)_{j\in \NN}\subset Z$ 
such that 
$( \| \psi_j \|_Z)_{j \in \NN} \in \ell_1(\NN)$.
Set
\begin{equation}\label{eq:u-sigma}
	\sigma(\by) := \sum_{j\in \NN} y_j\psi_j 
	\quad \mbox{and} \quad 
	u(\by) := \frku(\sigma(\by)) \;,
	\quad 
	\by\in \IIi
	\;.
\end{equation}
We shall work under
\begin{assumption}\label{ass:Hol}
	$X,Z$ are complex Hilbert spaces. 
	The sequence $(\psi_j)_{j\in \NN}\subset Z$,
	and there are real numbers $p\in (0,1]$, $r>0$ and $M>0$ 
	such that, with $B^Z_r(\phi)\subset Z_\CC$ 
	denoting the open ball of radius $r$ centered at $\phi\in Z$,
	\begin{enumerate}
		\item[{\rm (i)}]
		$\bb = (\| \psi_j \|_Z)_{j \in \NN} \in \ell_p(\NN)$,
		\item[{\rm (ii)}] 
		with $\sigma(\by) = \sum_{j\in \NN} y_j\psi_j \in Z$, 
		and with the sets 
		$$
		K := \{ \sigma(\by): \by \in \IIi \} \;\;\mbox{and} \;\;
		S_K := \bigcup_{\phi \in K} B^Z_r(\phi) \subseteq Z
		$$
		it holds 
		$\frku \in \Hol(S_K;X)$,
		\item[{\rm (iii)}] 
		$\sup_{z\in S_K} \| \frku(z) \|_X = M < \infty$.
		\item[{\rm (iv)}]
		The function $u:\IIi \to X$ is given in terms of $\frku$ as in \eqref{eq:u-sigma}.
	\end{enumerate}
\end{assumption}
Jacobi series approximation convergence rates 
of holomorphic in $\II$ functions $u$ 
are well known to be related to 
the classical \emph{Bernstein ellipse} $\calE_\rho\subset \CC$ 
with foci at $z = \pm 1$ and semiaxis-sum $\rho > 1$.
For a sequence 
$\brho =  (\rho_j)_{j \in \NN} \in (1,\infty)^\NN$
of semiaxis-sums, define 
$\calE_\brho := \calE_{\rho_1} \times \calE_{\rho_2} \times \ldots \subset \CC^\NN$.
We collect some elementary properties of maps 
$\by \mapsto u(\by)$ obtained from a holomorphic $\frku:Z\to X$ 
in Assumption~\ref{ass:Hol}, parameterized as in \eqref{eq:u-sigma}.
\begin{proposition}\label{prop:u(y)}
	Let $u:\IIi\to X$ be as defined in \eqref{eq:u-sigma} with $\frku$ satisfying
	Assumption~\ref{ass:Hol}. 
	Set $\bb = (b_j)_{j\in \NN}$ with $b_j := \| \psi_j \|_Z$.
	Then there holds
	\begin{enumerate}
		\item[{\rm (i)}]
		$u:\IIi \to X$ is continuous,
		\item[{\rm (ii)}]
		for every sequence $\brho \subset (1,\infty)$ which is $(\bb,r)$-admissible, i.e.,
		\begin{equation*}\label{eq:bradm}
			\sum_{j\in \NN} b_j(\rho_j-1) \leq r \;,
		\end{equation*}
		$\by\mapsto u(\by)$ allows a separately holomorphic extension to $\calE_\brho$ 
		(denoted also with $u$) ,
		\item[{\rm (iii)}] 
		with
		$$
		S_{\bb,r} 
		:= 
		\bigcup_{\{ \brho : \brho \mbox{ is } (\bb,r)-\mbox{admissible} \} } \calE_\brho
		\subset 
		Z
		$$
		the extension $u:S_{\bb,r}\to X$ is well-defined and 
		\begin{equation*}\label{Mbound}
			\sup_{z\in S_{\bb,r}} \| u(z) \|_X \leq M <\infty\;.
		\end{equation*}
	\end{enumerate}
\end{proposition}
For a proof, we refer to e.g. \cite[Lemma 2.2.7]{Zec18T}.
There holds the following summability of the 
Jacobi gpc coefficients $u_\bnu$ in the gpc series \eqref{J-series}
of $\by\mapsto u(\by)$ in \eqref{eq:u-sigma}.
\begin{theorem}\label{thm:JacBd}
	Consider the parametric function $u:\IIi\to X: \by\mapsto u(\by)$ 
	obtained from a holomorphic map $\frku$ with the input-encoding \eqref{eq:u-sigma},
	so that Assumption~\ref{ass:Hol} holds for the resulting parametric map $\by\mapsto u(\by)$.
	Let $r>0$, $p>0$ and $\bb \in (0,1]^\infty \cap \ell_p(\NN)$.
	
	Then, 
	with the weight $p_\bnu := p_\bnu(\theta,\lambda)$ 
	for $\bnu\in \FF$, with $p_\bnu(\theta,\lambda)$ as defined in \eqref{[p_s]} 
	for arbitrary given $\theta, \lambda > 0$, 
	there exists $C>0$ and, for each $\kappa \in \NN$, 
	a monotonically decreasing sequence 
	$(a_{\bnu})_{\bnu\in \FF} \in (0,\infty)^\infty$ 
	such that 
	\begin{enumerate}
		\item[{\rm (i)}] 
		$(a_{\bnu})_{\bnu\in \FF_\kappa} \in \ell_{p/\kappa}(\FF_\kappa)$,
		\item[{\rm (ii)}] 
		for $u:\IIi \to X$ related to $\frku$ as in Assumption~\ref{ass:Hol}, 
		for $M,r>0$ as above, and 
		for $(\psi_j)_{j\in \NN} \subset Z$ with $\| \psi_j \|_Z \leq b_j$ for all $j\in \NN$, 
		there exists $C>0$ such that 
		the Jacobi coefficients $(u_\bnu)_{\bnu\in \FF}$ of $u$ satisfy
		\begin{equation}\label{eq:JacuBd}
			\forall \bnu\in \FF_\kappa: \quad 
			{p_\bnu} \| u_\bnu \|_X \leq C M a_{\bnu}.
		\end{equation}
		In particular, for every $\kappa \in \NN$ holds
		$$
		(p_\bnu \|  u_\bnu \|_X )_{\bnu\in \FF_\kappa}\in \ell_{p/\kappa}(\FF_\kappa).
		$$
	\end{enumerate}
\end{theorem}
This result is contained in \cite[Thm.~2.2.10]{Zec18T}, where a complete proof is 
available. 
As in earlier works \cite{CDS11,BCM17,CCS15,ADMHolo},
it uses complex-variable methods to obtain precise bounds on the Jacobi-coefficients $u_\bnu$ 
expressed via the Cauchy-Integral Theorem on suitable contours in poly-ellipses $\calE_\brho$
with $(\bb,r)$-admissible semi-axis sums $\brho$ as in Proposition~\ref{prop:u(y)}. 

We relate Theorem~\ref{thm:JacBd} to the abstract sparse-grid interpolation and quadrature 
results, Corollary~\ref{corollary: FullyPolApprEv} and Theorem~\ref{thm[quadrature]}.
To this end, we verify that the parametric holomorphy in Proposition~\ref{prop:u(y)} 
implies the double-weighted summability Assumption~\ref{assum2}.
\begin{corollary}
	\label{cor:doubSum}
	Under the assumption and notation of Theorem \ref{thm:JacBd},
	let  $0< p/\kappa <2$, \ $\theta, \lambda >0$, and define the sequence
	$\bsigma:=(\sigma_{\bs})_{\bs \in \FF}$ by
	\begin{equation*} 
		\sigma_{\bs}:= 	a_{\bs}^{p/2\kappa - 1} p_\bnu(\theta,\lambda),  \ \bs \in \FF.
	\end{equation*}	
	Then there exists a constant $M_\kappa>0$ such that
	\begin{equation} \label{ell_2-summability}
		\left(\sum_{\bs\in\FF_\kappa} (\sigma_{\bs} \|u_\bs\|_X)^2\right)^{1/2} \ \le M_\kappa^{1/2} \ <\infty 
		\ \ \text{and} \ \
		\norm{\bp(\theta,\lambda) \bsigma^{-1}}{\ell_{q/\kappa}(\FF_\kappa)} \le M_\kappa^{\kappa/q} < \infty,
	\end{equation}	
	where 
	$q_\kappa := \frac{2p}{2-p/\kappa}$ 
	and 
	$M_\kappa:= \norm{\ba}{\ell_{p/\kappa}(\FF_\kappa)}^{p/\kappa}$.
\end{corollary}	

\begin{proof}
	We prove the corollary for the particular case when $\kappa =1$. 
	The case when $\kappa \ge 2$
	can be proven in an analogous manner with obvious modifications.
	
	Observe $\FF = \FF_1$ and that $M_1 < \infty$ and 
	$$
	\|u_\bs\|_X \le  p_\bnu(\theta,\lambda)^{-1} a_{\bs}
	$$ 
	by \eqref{eq:JacuBd} in Theorem~\ref{thm:JacBd}.		
	Hence, we have by Theorem~\ref{thm:JacBd},
	\begin{equation*}  
		\sum_{\bs\in\FF} (\sigma_{\bs}\|u_\bs\|_X)^2
		\ \le \
		\sum_{\bs\in\FF} 
		\brac{p_\bnu(\theta,\lambda) a_{\bs}^{p/2 - 1}
			p_\bnu(\theta,\lambda)^{-1} a_{\bs}}^2
		\ = \
		\norm{\ba}{\ell_p(\FF)}^p 	
		\ \le \ M_1,
	\end{equation*}	
	and 
	\begin{equation*} 
		\big\|\bp(\theta,\lambda)  \bsigma^{-1}\|_{\ell_q(\FF)}^q
		\ = \ 
		\sum_{\bs\in\FF} 
		\brac{a_{\bs}^{1- p/2}}^{2p/(2-p)} 
		\ = \
		\norm{\ba}{\ell_p(\FF)}^p 
		\ \le \
		M_1,
	\end{equation*}	
	which proves \eqref{ell_2-summability} for $\kappa = 1$.
	\hfill
\end{proof}

The general results, 
Theorem~\ref{thm[quadrature]} and Corollary~\ref{corollary: FullyPolApprEv},
on convergence rates of fully discrete sparse-grid interpolation and quadrature 
imply with the double summability in Assumption~\ref{assum2} (which in the presently
considered case is a consequence of Theorem~\ref{thm:JacBd} and Corollary~\ref{cor:doubSum} with $\kappa=2$)
the following result valid under the parametric holomorphy 
$\frku: Z^i\to X^i$, $i=1,2$.
\begin{theorem}\label{thm:HolSpGQ}
	For complex Hilbert spaces $X^i,Z^i$, $i=1,2$, 
	with $X^2\subset X^1$ and $Z^2\subset Z^1$,
	assume given a map $\frku:Z^i\to X^i$, $i=1,2$ 
	which is in each case holomorphic according to 
	Assumption~\ref{ass:Hol}, item (ii), with suitable exponents $0<p_1\leq p_2 <2$,
	and corresponding sequences $\bb_i$, $i=1,2$ as in Assumption~\ref{ass:Hol}, item (i).

	Assume the parametric maps $\IIi \ni \by \mapsto u(\by)\in X^i$ 
	are given in terms of $\frku$ via affine encoding $\sigma$ in \eqref{eq:u-sigma} 
	with one (common) sequence $(\psi_j)_{j \in \NN} \subset Z^2\subset Z^1$ as in Assumption~\ref{ass:Hol},
	satisfy the double weighted summability 
        Assumption~\ref{assum2}, item (iv)
	with the summability exponents $q_i:= 2p_i/(2-p_i)$, $i=1,2$. 

        Let $\bsigma_i:=(\sigma_{i;\bs})_{\bs \in \FF}$, $i=1,2$,
	be the sets defined as in Corollary~\ref{cor:doubSum} in the context of the space $X^i$.
	
	Assume further given a sequence $(X_m)_{m \in \NN}$ of subspaces $X_m\subset X^1$ 
	such that \eqref{SpatialAppr} in 
        Assumption~\ref{assum2}, item (iii) holds with rate $\alpha > 0$.
	\begin{enumerate}
	\item[{\rm (i)}]
	Then there hold the interpolation error bounds \eqref{vSGxi3}, 
	with the rate $\alpha>0$ which is as in \eqref{SpatialAppr} of 
        Assumption~\ref{assum2}, item (iii)
	and the rates $\beta$ and $\kappa$ are as in \eqref{beta}.
	\end{enumerate}

	\medskip
	Assume in addition that the product Jacobi measure \eqref{eq:ProdJPM} is symmetric, i.e.,
	that $a_i = b_i$ for all $i \in \NN$. 
        For $\xi > 1$, let $\tilde{G}_\rev(\xi)$ be the set defined as in  
	Definition~\ref{def:tG_rev(xi)} for $\bsigma_i$, $i=1,2$, 
	as in Assumption~\ref{assum2}, item (iv). 	
		
	Then for  the quadrature operator 
	$\Qq_{\tilde{G}_\rev(\xi)}$ 
	generated by the interpolation operator $\Ii^*_{\tilde{G}_\rev(\xi)}: C(\IIi, X^2)\to \Vv(\tilde{G}_\rev(\xi))$, 
	we have the following.
		\begin{itemize}
			\item[{\rm (ii)}]
			There exists a constant $C>0$ such that 
			for any $n\in \NN$ there exists a number $\xi_n$ 
                        such that
			$\dim\Vv(\tilde{G}_\rev(\xi_n))\le n$ 
                        and 
			\begin{equation*} \label{u-Q_Gu-quadratureJ2}
				\left\|\int_{\IIi}u(\by)\, \rd \mu(\by) - \Qq_{\tilde{G}_\rev(\xi_n)}u\right\|_{X^1} \leq 	C
				\begin{cases}
					n^{-\alpha} &\text{if } \alpha \leq  2/q_2 - 1/2,
					\\
					n^{-\beta}(\log n)^\kappa &\text{if } \alpha >  2/q_2 - 1/2.
				\end{cases}  
			\end{equation*}			
			\item[{\rm (iii)}] 
			Let $\phi \in (X^1)'$  be a bounded linear functional on $X^1$. 
			There exists a constant $C>0$ such that 
			for any $n\in \NN$ there exists a number $\xi_n$ such that $\dim\Vv(\tilde{G}_\rev(\xi_n))\le n$ and 
			\begin{equation*} \label{u-Q_Gu_phiJ2}
				\left|\int_{\IIi} \langle \phi,  u(\by) \rangle \rd \mu(\by) - \Qq_{\tilde{G}_\rev(\xi_n)} \langle \phi,  u \rangle\right| \leq C \|\phi\|_{(X^1)'}
				\begin{cases}
					n^{-\alpha} &\text{if } \alpha \leq  2/q_2 - 1/2,
					\\
					n^{-\beta}(\log n)^\kappa &\text{if } \alpha >  2/q_2 - 1/2.
				\end{cases} 
			\end{equation*}
		\end{itemize}	
		The rate $\alpha$ is given by~\eqref{SpatialAppr}, $\kappa$ and $\beta > 0$ 
                by \eqref{eq:beta2} with $q_r$ being replaced by $q_2$. 	
	\end{theorem}
\begin{proof}
Item (i) is directly obtained by applying Theorem~\ref{thm: FullyPolAppr}. 
Items (ii) and (iii) can be proven in a manner similar to the proof of Theorem~\ref{thm[quadrature]pdeJ}.
\hfill
\end{proof}

\begin{remark}
	{\rm
	We compare the convergence rates of fully discrete sparse-grid polynomial interpolations 
and quadratures in Theorems~\ref{thm[coll-approx]pdeJ}--\ref{thm[quadrature]pdeJ} and in Theorem~\ref{thm:HolSpGQ}. 
Denote by {$A_n$ and $B_n$} the bounds (without constants)  
for these convergence rates as in the right-hand sides of 
\eqref{vSGxi3Rev2} and \eqref{u-Q_Gu-quadratureJ}, respectively. 
Evidently, ignoring values of constants,
$A_n$ and $B_n$ cannot exceed $n^{-\alpha}$ 
which is governed by the spatial regularity $\alpha$ 
in Assumption~\ref{assum2}, item~(iii).
By simple computation we derive that
		\begin{equation*} 
			\begin{cases}
				B_n = A_n = n^{-\alpha}  \  &\text{if }  \ \alpha \leq  1/q_2 - 1/2,
				\\
				B_n = A_n n^{-\tau_1} (\log n)^{-\kappa}  \  &\text{if }  \ 1/q_2 - 1/2 < \alpha \leq  2/q_2 - 1/2,
				\\
				B_n = A_n n^{-\tau_2}    \ &\text{if } \ \alpha >  2/q_2 - 1/2,
			\end{cases}
		\end{equation*}	
for	
$$
\tau_1:= \frac{\alpha(\alpha -1/q_2 +1/2)}{\alpha +\delta}>0,\qquad  
\tau_2:= \frac{\alpha(\alpha/q_1 +\delta /2)}{(\alpha +\delta)(\alpha +2 \delta)}>0,
$$
 where   $\delta:= 1/q_1 - 1/q_2 \ge 0$. 
 This shows a dichotomy between the asymptotic convergence rates, i.e.,  the behaviors of 
 $A_n$ and $B_n$,
 which depends on the relation of spatial regularity $\alpha$ and weighted summability exponent $p_2$.
More precisely,
$A_n$ dominates $B_n$ 
in the case of higher spatial regularity when $\alpha >  1/q_2 - 1/2$.
In the complementary case, 
i.e., in the case of lower spatial regularity when $\alpha  \le1/q_2 - 1/2$,
both asymptotic rates coincide and equal $n^{-\alpha}$.
As noted earlier, this principal improvement in the first case 
stems from the cancellation of anti-symmetric terms within the sparse-grid tensor-product quadratures associated with the symmetric  
infinite-tensor-product ultra-spherical polynomials.
}
\end{remark}
%
%

\end{document}